\documentclass[11pt]{article}
\usepackage[T1]{fontenc}
\usepackage{url}
\usepackage{graphics}
\usepackage{amsthm}
\usepackage{amssymb}
\usepackage{mathtools}
\usepackage{stmaryrd}	
\SetSymbolFont{stmry}{bold}{U}{stmry}{m}{n}	% To avoid warnings
\usepackage{bm}
\usepackage{mathrsfs}
\usepackage{euscript}
\usepackage{tikz-cd}
\usetikzlibrary{positioning, shapes.geometric, patterns, graphs, decorations.pathmorphing}
\usepackage{ulem}
\usepackage{subfigure}

\usepackage{paralist}

\usepackage{color}
\definecolor{Mahler}{RGB}{78,124,161}
\definecolor{Prokofiev}{RGB}{139,0,18}
\definecolor{Shostakovich}{RGB}{20,74,116}
\usepackage[colorlinks]{hyperref}
\hypersetup{
    colorlinks=true,
	linkcolor=Mahler,
	citecolor=Prokofiev,
	urlcolor=Shostakovich,
	pdftitle = {},
	pdfauthor = {Yuyang Feng}    bookmarks=true,    bookmarksopen=true,
    bookmarksopenlevel=2,
    unicode=true,
    linktocpage,
}

\theoremstyle{plain}
\newtheorem{proposition}{Proposition}
\newtheorem{lemma}[proposition]{Lemma}
\newtheorem{theorem}[proposition]{Theorem}
\newtheorem{corollary}[proposition]{Corollary}
\theoremstyle{definition}
\newtheorem{definition}[proposition]{Definition}
\newtheorem{definition-theorem}[proposition]{Definition--Theorem}
\newtheorem{definition-proposition}[proposition]{Definition--Proposition}
\newtheorem{remark}[proposition]{Remark}

\newtheorem{conjecture}[proposition]{Conjecture}

\newtheorem{thm}[proposition]{Theorem}

\theoremstyle{definition}

\theoremstyle{plain}

\numberwithin{equation}{section}
\numberwithin{proposition}{section}
\numberwithin{conj}{section}

\setdefaultitem{$\star$}{\normalfont\bfseries \textendash}{\textasteriskcentered}{\textperiodcentered}	

\usepackage{amsmath}
\numberwithin{figure}{section}
\numberwithin{table}{section}

\usepackage{geometry}
\usepackage{lmodern}

\renewcommand{\P}{\mathbb{P}}
\newcommand{\eqd}{\overset{d}{=}}
\newcommand{\E}{\mathbb{E}}

\newcommand{\lt}{s}
\newcommand{\rt}{t}

\newcommand{\T}{T}
\newcommand{\M}{M}
\newcommand{\G}{G}
\newcommand{\scale}[2][n]{#2^{(#1)}}
\newcommand{\proj}{\hyperref[def:proj]{\mathcal{P}}}
\newcommand{\scaleproj}{\hyperref[def:scaleproj]{\mathcal{P}_n}}

\newcommand{\CRT}{\hyperref[def:CRT]{\mathcal{T}}}
\newcommand{\lebT}{\lambda_{T}}

\newcommand{\dT}{d_{T}}
\renewcommand{\d}{d}
\newcommand{\dM}{d_{M}}
\newcommand{\dCRT}{d_{\CRT}}
\newcommand{\lebCRT}{\lambda_{\CRT}}
\newcommand{\oCRT}{o_{\CRT}}

\title{Triviality of critical Fortuin-Kasteleyn decorated planar maps for $q>4$}
\author{\begin{tabular}{c} Yuyang Feng\\[2.5pt] \textit{University of Chicago }\end{tabular}}
\date{}

\begin{document}
	\maketitle
	\begin{abstract}
		We consider infinite random planar maps decorated by the critical Fortuin-Kasteleyn model with parameter $q>4$. The paper demonstrates that when appropriately rescaled, these maps converge in law to the infinite continuum random tree as pointed metric-measure spaces, that is, with respect to the local Gromov-Hausdorff-Prokhorov topology. Furthermore, we also show that these maps do not admit any Fortuin-Kasteleyn loops with a macroscopic graph distance diameter. Our proof is based on Scott Sheffield's hamburger-cheeseburger bijection.
	\end{abstract}
	
	\normalem
	
	%{\scriptsize
		%\begin{tabular}{ll}
		%	\textbf{MSC (2010)} & 11F70; 11R58 22E55 \\
		%	\textbf{Keywords} &
		%\end{tabular}}
		
	\tableofcontents
	%-----------------
 
    \section{Introduction}\label{Chap_Introduction}        
		In the past few decades, there has been extensive research concerning random planar maps. A planar map is a graph embedded into the Riemann sphere in such a way that no two edges cross, viewed modulo orientation-preserving homeomorphisms. Self-loops and multiple edges are allowed. The study of planar maps dates back to the 1960s with Tutte’s attempts at the four-color problem \cite{Tuttetri} \cite{Tutte2} \cite{tutte3} \cite{tutte4}.
		
		One of the motivations for studying random planar maps comes from statistical mechanics. Certain random planar maps are discrete analogs of $\gamma$-Liouville quantum gravity (LQG) surfaces with $\gamma\in (0,2]$, which are random topological surfaces equipped with a measure, a metric, and a conformal structure. Such surfaces play a vital role in many physics models, starting from the foundational work of Polyakov \cite{physicsorigin}. The earliest mathematical studies include \cite{KPZformula} and \cite{LQGandKPZ}. We will not use the mathematical theory of LQG in this paper. For more background on LQG, readers can refer to \cite{gfflqggmc} \cite{gwylqg} \cite{sheICM}. 
		
		This paper will focus on the random planar map decorated by the critical Fortuin-Kasteleyn (FK) model with parameter $q>0$ \cite{Original_paper_of_FK}\cite{she16}. This model can be thought of as a discrete version of 2D gravity in which the geometry of the space interacts with the matter. The critical FK model has a deep relationship with the critical $q$-state Potts model. In fact, their partition functions are equal, see e.g. \cite{equivalence_potts_FK}.
  
        For a planar map $\M$, a spanning subgraph $\G$ is formed by all of the vertices of $\M$ and a subset of the edges of $\M$. The \textbf{dual subgraph} $G^*$ is defined in the following way: We assign each face $f$ of $M$ a vertex $f^*$. When a common edge $e$ of adjacent faces $f$ and $h$ is not in $G$, we draw an edge $e^*$ crossing $e$ and connecting $f^*$ and $h^*$. The vertex set of $G^*$ is given by $\{f^*\}$ and edge set is given by $\{e^*\}$. In fact, $G^*$ is also a spanning subgraph of the dual map of $M$. Next, we draw lines from each $f^* \in G^*$ to each corner of $f$, connecting the vertices of $f$ in $G$. These edges are called \textbf{Tutte edges}. Consider the interfaces between clusters in $G$ and clusters in $G^*$, they will be loops crossing the Tutte edges and we call them \textbf{FK loops}. Let  $\ell(\M,\G)$ be the number of FK loops. Note $\ell(\M,\G) \geq 1$ with equality iff $G$ is a spanning tree. 

         We say a map $M$ is \textbf{rooted} if it is decorated by a distinct Tutte edge. Note this Tutte edge has exactly one endpoint in $M$. We call it the root of the rooted map $M$.
		
		\begin{figure}[htbp]
			\centering
			\vspace{0.2in}
			\includegraphics[scale=0.27]{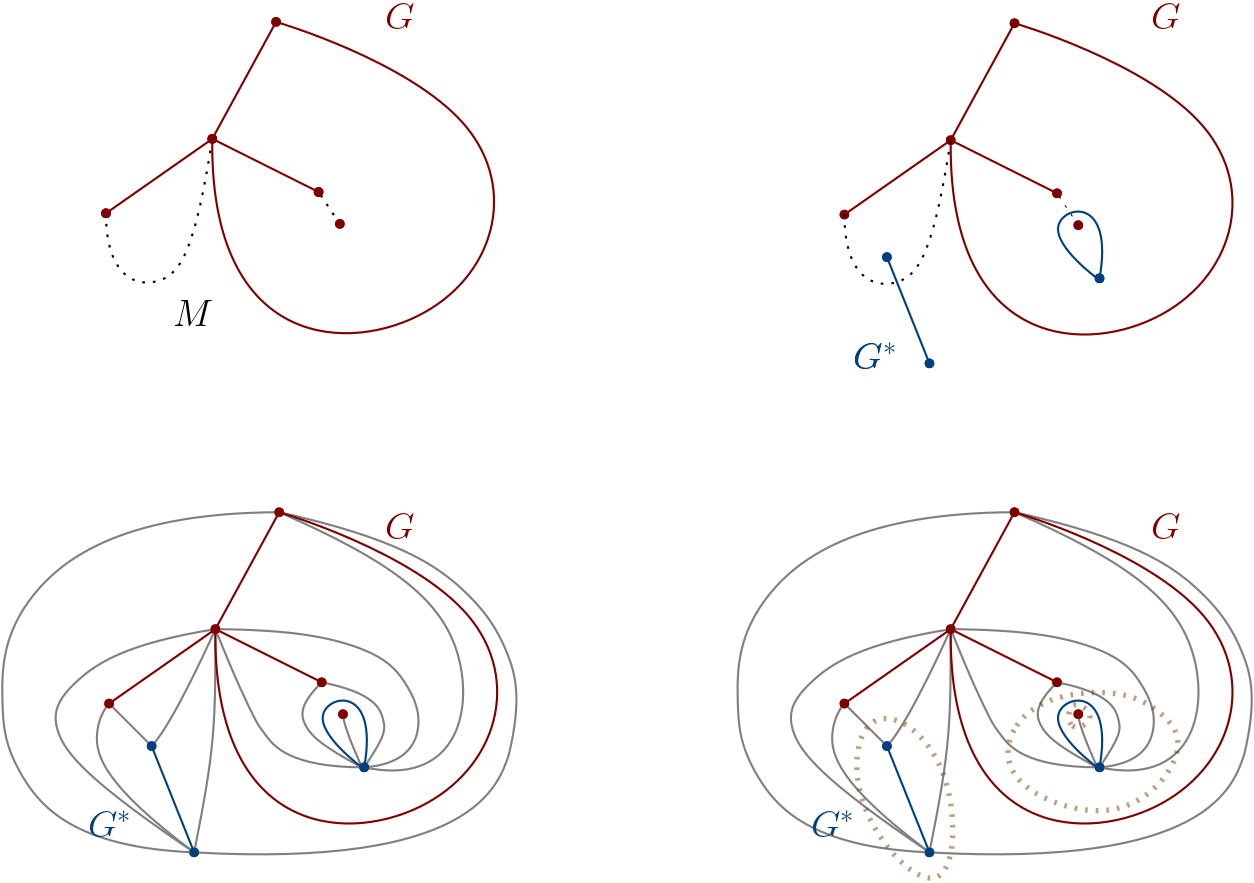}
			\caption{\textbf{Top left:} The spanning subgraph $\G$ is colored red, dotted lines represent the edges of $\M$ which are not in $\G$. \textbf{Top right:} The dual subgraph $\G^*$ is colored blue. \textbf{Bottom left:} The Tutte edges are colored gray. \textbf{Bottom right:} Dotted brown circles are the FK loops, which cross some Tutte edges.}
			\label{fig:Tutte}
		\end{figure}
		
		\begin{definition} \label{def_FK map} 
			Let $\mathcal{M}_n$ be the set of pairs $(M,G)$, where $M$ is a rooted map with $n$ edges, $G$ is a spanning subgraph. For $q \in (0,\infty)$, the law $P_{n,q}$ of the \textbf{critical $q$-FK planar map} is the probability measure on $\mathcal{M}_n$ such that:
			\[
			P_{n,q}(\M, \G) \propto \sqrt{q}^{ \ \ell(\M, \G)}
			\]
			By taking appropriate limits, we can also define $P_{n,q}$ for $q \in \{0,\infty\}$. 
		\end{definition}

        From the definition, we can see $P_{n,q}$ is self-dual, that is, $(M,G) \overset{d}{=} (M^*,G^*)$. This is why we call this law ``critical''. This model is also related to the critical $q$-state Potts model \cite{she16}.
		
		The $q$-FK planar map has a natural interpretation for some special values of $q$. For example, when $q=0$, the law is the uniform measure on pairs consisting of a planar map and one of its spanning trees.
		When $q=1$, since each planar map with $n$ edges has exactly $2^n$ spanning subgraphs, the marginal law of $\M$ is uniform. 
		When $q=2$, the law of $\G$ is the FK representation of the critical Ising model. See Figure \ref{fig:simulation} for computer simulations of FK planar maps.
		
		As explained in the Appendix of \cite{she16}, the scaling limits of finite FK planar maps are conjectured to be as follows. We deliberately do not specify the particular topology of convergence here, but see the discussion below for one possible topology.
		
		\begin{conjecture}\label{conj:q>4}
			When $0 \leq q\leq4$, the $q$-FK planar map converges in the scaling limit to a $\gamma$-LQG surface, where $\gamma\in [\sqrt 2, 2)$ satisfies
			\[
			q=2+2 \cos \left(\frac{\gamma^{2} \pi}{2}\right).
			\]
			When $q > 4$, the $q$-FK planar map converges in the scaling limit to the continuum random tree \cite{CRT}.
		\end{conjecture}
		
		\begin{figure}[htbp]
			\centering
			\subfigure{
				\includegraphics[width=0.45\textwidth]{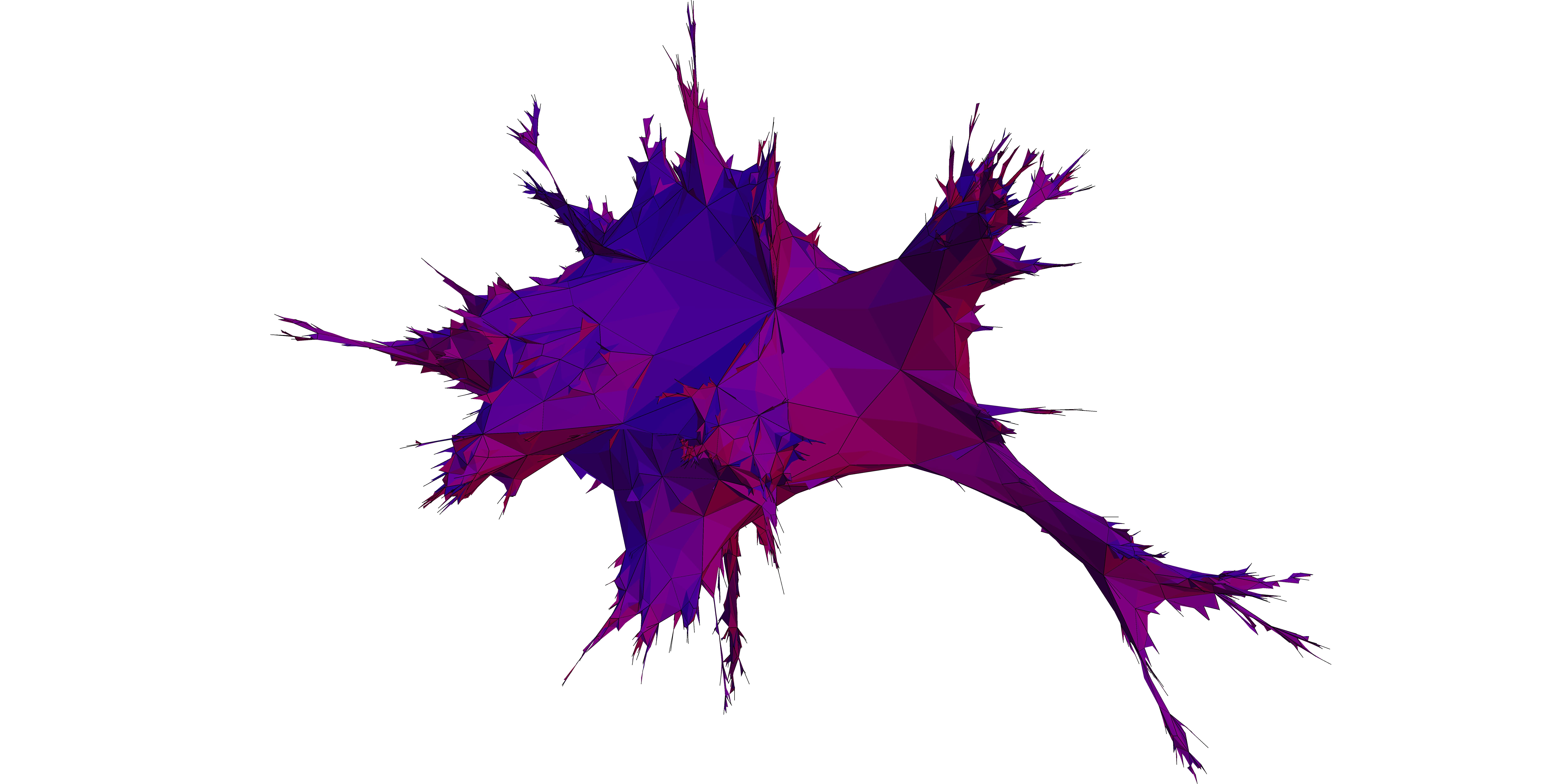}}
			\subfigure{ \includegraphics[width=0.45\textwidth]{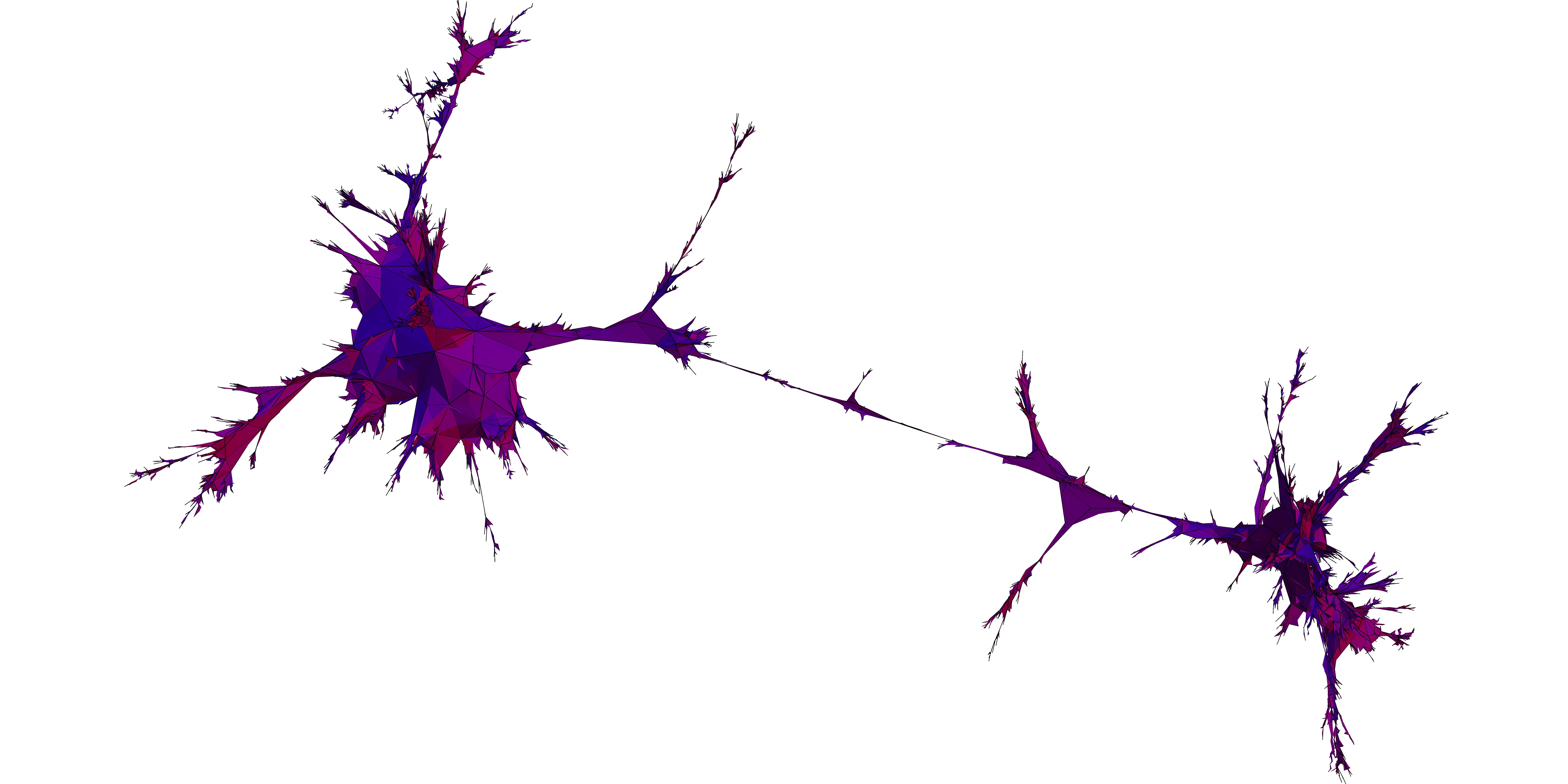}}
			\caption{FK planar maps with $q = 1/2$ (left) and $q=9$ (right) drawn in $\mathbb R^3$, simulated by  \href{http://www.normalesup.org/~bettinel/simul_FK.html}{Jérémie Bettinelli}} 
			\label{fig:simulation}
		\end{figure}

         The Gromov-Hausdorff-Prokhorov (GHP) distance is a distance between \textit{pointed} metric-measure spaces. To view FK maps as pointed metric-measure spaces, we can endow them with graph metrics $d$, measure $\mu$ that give each vertex a mass equal to its degree, and the root vertex $o$ as the specific point. From the definition of FK map, when condition on the base map, the root edge is uniformly selected, thus $o$ is selected according to $\mu$.
		
		This paper will focus on the case when $q > 4$. Our main result is the following theorem, establishing a version of Conjecture~\ref{conj:q>4} in the case $q\leq 4$ in the infinite-volume setting.
		\begin{theorem}\label{main_thm_in_introduction}
			When $q>4$, there is a deterministic constant $\hyperref[def:alpha]{\mathfrak{a}}=\hyperref[def:alpha]{\mathfrak{a}}(q)$, such that as $\varepsilon$ tends to $0$, the rescaled infinite FK map $(M, \varepsilon \hyperref[def:alpha]{\mathfrak{a}}d,\varepsilon^2 \mu, o)$ converges to the infinite continuum random tree in distribution w.r.t.\ the local Gromov-Hausdorff-Prokhorov distance.
		\end{theorem}

        The \textbf{infinite $q$-FK planar map} is the Benjamini-Schramm local limit \cite{local_limit} of the finite $q$-FK planar map. Roughly speaking, for any $r\in \mathbb{Z}_{\geq 0}$, the law of $r$-ball around the root vertex of this map is the limit of the law of the $r$-ball in finite FK map. See \cite{che17} for proof of convergence. The infinite FK map also has an alternative description using random walks which we shall see in Chapter \ref{Chap_Bijection}.  
		The infinite FK map in the $q \leq 4$ case has been studied in \cite{FK_exponents} and \cite{FK1}\cite{FK2}\cite{FK3}, which computed exponents associated with the length and the loop area, etc.\ and described the scaling limits of some of these quantities.	

		The \textbf{infinite continuum random tree} (CRT) is an infinite version of classical CRT in \cite{CRT}, encoded by a two-sided Brownian motion starting at zero. The \textbf{local GHP distance} is a local version of GHP distance, which applies to locally compact pointed space endowed with a locally finite measure. We will see the exact definitions in Chapter \ref{Chap_Tree_to_CRT}.
		
		\begin{remark}
			With more work, we can also prove $(M, \varepsilon \hyperref[def:alpha]{\mathfrak{a}}d,\varepsilon^2 \mathfrak{b}\mu_v, o)$ converges to the CRT, where $\mu_v$ is the counting measure on vertices. Here, $\mathfrak{b}$ is the expectation of the root's degree in $M$. The degree of the root is proven to have an exponential tail distribution by \cite[Proposition 6]{che17}, hence $\mathfrak{b}$ is finite.
		\end{remark}
		
		The FK loops in the subcritical and critical cases are believed to converge to the \textbf{conformal loop ensemble} with parameter $\kappa \in [4,8)$ (see e.g., \cite{Scott_CLE} and the Appendix of \cite{she16}), where
		$$q=2+2 \cos \left(\frac{8\pi}{\kappa}\right).$$        
		In this paper, we will see that in the supercritical case, the FK loops will vanish via the scaling process in Theorem \ref{main_thm_in_introduction}. To be more precise, 
		
		\begin{theorem}\label{thm:main-chap5-loop}
			Define the diameter of the FK loop to be the max $M$-distance between the endpoints of Tutte edges it crosses. Let $S_n$ be the collection of FK loops that intersect\footnote{We say an FK loop intersects a vertex set $B$ if it crosses a Tutte edge with an endpoint in $B$. } the graph distance ball of radius $n$ around the root vertex. Then 
			$$\sup_{\ell \in S_n}  \frac{1}{n}\operatorname{diam} \ \ell \rightarrow 0, \quad \text{in probability.}$$
            where $\operatorname{diam} \ell$ denotes the greatest distance between any vertices of $M$ which are endpoints of Tutte edges intersected by $\ell$.
		\end{theorem}

        \begin{remark}
             Another topology with respect to which convergence of random planar maps has been studied in \cite{she16}. In short, we can glue together, in certain sense, a pair discrete plane trees associated with a certain pair of random walks on $\mathbb{Z}$, to encode maps $\{M_i\}$ (see chapter \ref{Chap_Bijection} for details). When the walks converges to correlated continuous functions which encode map $M$, we say $M_i \to M$ in \textbf{mating of trees} topology. In \cite{she16}, it is proved the random walks which encode FK maps will converge to correlated Brownian motion in distribution when $q < 4$, and converge to identical Brownian motion when $q \geq 4$. The trees that Brownian motion associated with are CRTs. Gluing together two identical CRT just gives back to the original CRT. We also know from the monumental paper \cite{DMS14} that we can glue together the CRTs associated with two correlated Brownian motions to get an LQG surface. Thus, paper \cite{she16} proves the conjecture \textit{in mating of trees sense} when $q \neq 4$. However, mating of trees convergence is far from being equivalent to GHP convergence. For example, when $q=4$, the mating of trees convergence says they should converge to CRT. But in GHP sense, the conjecture is that $q$-FK planar maps converge to critical LQG.
        \end{remark}

        \begin{remark}
           For many random planar map models, the planar maps are expected or proven to converge to the CRT for parameter values in which the model is, in some sense, ``degenerate". See the lecture notes \cite{curienlec} for some examples. Also, for the maps that belong to the universality class of \textit{supercritical} LQG, they are believed to converge to CRT when conditioned to be finite. See e.g. \cite{laplacian}\cite{Brownian_loops_and_the_central_charge_of_LQG}\cite{ang2023supercritical} for some discussions.
        \end{remark}

		The rest of the paper is organized as follows: In Chapter \ref{Chap_Bijection}, we will explain the hamburger-cheeseburger bijection in \cite{she16}, which encodes the infinite FK planar map by an infinite random walk and which will be the main tool in our proof. In Chapter \ref{Chap_Tree_to_CRT} we will give the construction of infinite CRT, and prove that it is the local GHP scaling limit of the spanning tree on the infinite FK planar map which appears in the hamburger-cheeseburger bijection. In Chapter \ref{Chap_map_to_tree}, we will prove the convergence of the FK map to the CRT. In the last chapter, we will see there is no macroscopic FK loop.

    \section*{Acknowledgements}
        The author wishes to express profound gratitude to Ewain Gwynne for suggesting this topic and for the insightful discussions that greatly inspired the author. Additionally, many thanks go to Jian Ding and Xinyi Li for funding the author's visit to Peking University, where part of this work was conducted.

	\section{Hamburger-Cheeseburger bijection}\label{Chap_Bijection}
		In this section, we will describe a bijection from random words to planar maps decorated by the FK model, which is known as the Hamburger-Cheeseburger bijection \cite{she16}. We will mainly describe the correspondences in the infinite case. They are similar to the finite case which has been discussed in \cite{she16}. 
		
		Recall the definitions above Figure \ref{fig:Tutte}, we can see the Tutte edges in fact form a quadrangulation\footnote{If a planar map with each face having 4 edges, we call it a quadrangulation}, which we denote by \textbf{$Q$}. Note for each face of $Q$, exactly one of its diagonals appears in either $G$ or $G^*$. From this point of view, we see $\hyperref[def:Tutte triangulation]{\mathbb{T}}:=Q\cup G \cup G^*$ as a canonical triangulation of $(M, G)$, and call it the \textbf{Tutte triangulation}\label{def:Tutte triangulation} of $(M, G)$. Note that $(M, G)$ and $\hyperref[def:Tutte triangulation]{\mathbb{T}}$ uniquely determine each other.
		
		Next, we introduce the inventory accumulation model in \cite{she16}. Imagine in a busy restaurant, the kitchen produces hamburgers and cheeseburgers, while the customers order these two kinds of burgers. We use $X=(X_n)_{n\in \mathbb{Z}}$ to record the burger-order sequence. Here, $X_n \in \{\texttt{a}, \texttt{b},\texttt{A},\texttt{B},\texttt{F}\}$, and the letters represent respectively, a hamburger, a cheeseburger, a hamburger order, a cheeseburger order, and a flexible order, that is, the order can be fulfilled by both $\texttt{a}$ and $\texttt{b}$. The burger-order stack follows the last-in-first-out (LIFO) rule, i.e., the customer always takes the freshest burger of the specified kind. For integers $s < t$, let $X(s,t)$ denote the subword $X_s X_{s+1} \dots X_{t-1} X_t$. The reduced form $\overline{X(s,t)}$ of $X(s,t)$ is obtained by applying the following relations:
		\[
		\texttt{a}\texttt{A}=\texttt{b}\texttt{B}=\texttt{a}\texttt{F}=\texttt{b}\texttt{F}=\emptyset, \quad  \texttt{a}\texttt{B}=\texttt{B}\texttt{a}, \quad \texttt{b}\texttt{A}=\texttt{A}\texttt{b}.
		\]
		In other words, $\overline{X(s,t)}$ consists of the unfulfilled orders and the unconsumed burgers between time $s$ and $t$. We say $s$ and $t$ are \textbf{matched} if the burger at time $s$ is consumed by the order at time $t$. We say $X(s,t)$ is \textbf{reducible}\label{def:reducible} if $\overline{X(s,t)}$ is empty. We can easily get the following property by the LIFO rule.
		
		\begin{lemma}[Locality]\label{lemma_locality}
			For $s<t$, whether $s$ and $t$ are matched only depends on $(X_i)_{i\in [s,t]}$. In particular, whether $X(s,t)$ is reducible only depends on $(X_i)_{i\in [s,t]}$.
		\end{lemma}
		
		Now we define a probability law for a bi-infinite word $X$. Let $p\in [0,1)$. We declare the symbols $X_n$ are independent samples from the distribution $\theta_p$ defined as follows:
		\[
		\theta_p(\texttt{a})=\theta_p(\texttt{b})=\frac{1}{4}, \quad \theta_p(\texttt{A})=\theta_p(\texttt{B})=\frac{1-p}{4}, \quad \theta_p(\texttt{F})=\frac{p}{2}.
		\]
		The law of $X$ is given by $\P_p={\theta_p}^{\otimes \mathbb{Z}}$.
		
		\begin{lemma}[\cite{she16}] For all $p \in [0,1]$, almost surely:
			\begin{itemize}
				\item[1.] Each burger (resp. order) has a matched order (resp. burger).
				\item[2.] For any time $t$, $\overline{X(-\infty,t)}$ contains infinitely many $\texttt{a}$ and infinitely many $\texttt{b}$.
			\end{itemize}
		\end{lemma}
		
		For a word with the above properties, we can use it to encode an infinite rooted planar map $\M$ decorated by a spanning subgraph $\G$. Moreover, $\M$ is locally finite \cite{che17}. Now we explain the construction.
		
		\begin{figure}[htbp]
			\centering
			\vspace{0.2in}
			\includegraphics[scale=0.22]{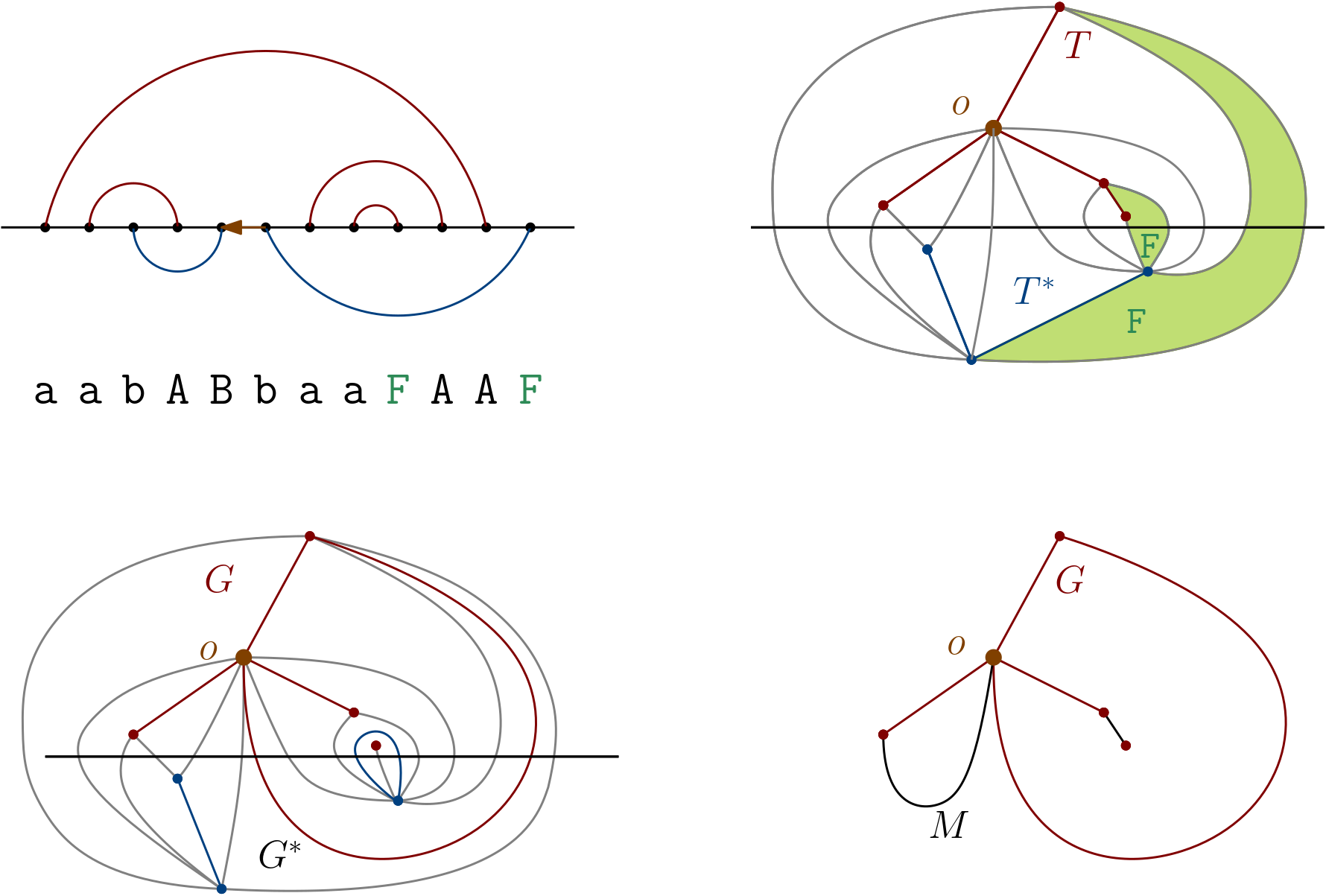}
			\caption{Hamburger-Cheeseburger bijection. Triangle faces corresponding to \texttt{F} are colored green.} 
			\label{fig:H-C bijection}
		\end{figure}
		
		We begin with the real axis with the vertex set $\mathbb{Z}$. For matched $i$ and $j$ ($i<j$), if $X_i$ is a hamburger (resp. cheeseburger), we draw a red (resp. blue) arc connecting $i$ and $j$ above (resp. below) the real axis. We get an arc map $\hyperref[arc map]{\mathbb{A}}$\label{arc map}, whose edge set is the union of the red arcs, the blue arcs, and the segments of the real line between consecutive points of $\mathbb Z$. See the top left in Figure \ref{fig:H-C bijection}. 
		
		In the second step, we take the dual map of $\hyperref[arc map]{\mathbb{A}}$ and color it as follows: Note every face of the arc map is either above or below the axis. For each face above the real axis, we put a red vertex inside it as the vertex in the dual map, and for the face below we put a blue one. For each red (resp. blue) arc, we draw a red (resp. blue) line crossing it as an edge in the dual map of $\hyperref[arc map]{\mathbb{A}}$. We color it gray for an edge in the dual map connecting a red point and a blue point, which will cross the real line. See the top right in Figure \ref{fig:H-C bijection}.
		
		Now we get a triangulation as the dual map of $\hyperref[arc map]{\mathbb{A}}$. Each triangle face contains exactly one $t \in \mathbb{Z}$, and is formed by two gray edges (that cross segments $[t-1,t]$ and $[t+1,t]$ respectively) and one red or blue edge. These gray edges are exactly the Tutte edges of $(M,G)$, and form the quadrangulation $Q$. For each quadrangular face of $Q$, one of its diagonals is colored either red or blue. We say the two diagonals of the same quadrangle are \textbf{dual}\label{def:dual_diagonal} to each other. Now we define the planar map $\M$ corresponding to our infinite word $X$ to be the map formed by the red edges, the red vertices, and the dual edges of the blue edges. The subgraph formed by red edges and red vertices is a spanning tree of $\M$, and we call this tree $\T=\T(\G)$ \label{def:canonical spanning tree} the \textbf{canonical spanning tree} of $\G$. In fact, the triangulation is exactly the Tutte triangulation $\hyperref[def:Tutte triangulation]{\mathbb{T}}$ of $(M,T)$.
		
		\label{Step3 of HC bijection} The final step is to obtain $(M,G)$ from $(M,T)$. we select all of the \texttt{F}-triangles (colored green in Figure \ref{fig:H-C bijection}), then flip the colored edge in this triangle to its dual edge. See bottom left of Figure \ref{fig:H-C bijection}. Now the new red graph is our $\G$. See the bottom left of Figure \ref{fig:H-C bijection} as the Tutte triangulation $\hyperref[def:Tutte triangulation]{\mathbb{T}}$ of $(\M,\G)$. See the bottom right of Figure \ref{fig:H-C bijection} for $(\M,\G)$.

        \vspace{0.1in}
		\noindent\textbf{Exploration line.} \label{exploration_line} For a triangle face in $\hyperref[def:Tutte triangulation]{\mathbb{T}}$ that contains $t \in \mathbb{Z}$, we denote it by \textbf{$\hyperref[def:Tri]{\mathrm{Tri}}(t)$}. Moreover, let \textbf{$\hyperref[def:Tri]{\mathrm{Tri}}[a,b]$} \label{def:Tri} be the planar submap $\subset \hyperref[def:Tutte triangulation]{\mathbb{T}}$ consisting of the union of the triangle faces $\{\hyperref[def:Tri]{\mathrm{Tri}}(t), t \in \mathbb{Z} \cap [a,b] \}$. We call $\hyperref[def:Tri]{\mathrm{Tri}}(t)$ a hamburger triangle (or a $\texttt{a}$-triangle) if $X(t)=\texttt{a}$, and so forth. The real axis is actually the timeline that we produce burgers and order burgers. A burger in time $s$ is consumed by an order in time $t$ iff $\hyperref[def:Tri]{\mathrm{Tri}}(s)$ and $\hyperref[def:Tri]{\mathrm{Tri}}(t)$ share the same red or blue edge (i.e they stay in the same quadrangle in $Q$).
		
		\vspace{0.1in}
		\noindent\textbf{Root.}
		We denote the collection of Tutte edges that cross $[a-1,b]$ in the real axis by \textbf{$\hyperref[def:V and E]{E}[a,b]$}\label{def:E}. For simplicity, we denote the only Tutte edge that crosses $[t-1,t]$ by \textbf{$\hyperref[def:V and E]{E}(t)$}. Each Tutte edge has one endpoint in $M$ (the red one) and one in the dual map of $M$ (the blue one). Let \textbf{$\hyperref[def:V and E]{V}[a,b]$}\label{def:V} (resp. $\hyperref[def:V and E]{V^*}[a,b]$) be the endpoints of $\hyperref[def:V and E]{E}[a,b]$ that are vertices in $M$ (resp. the dual map of $M$). Especially, let \textbf{$\hyperref[def:V and E]{V}(t)$} (resp. $\hyperref[def:V and E]{V^*}(t)$) be the endpoint of $\hyperref[def:V and E]{E}(t)$ that is a vertex of $M$ (resp. the dual map of $M$). We define Tutte edge $\hyperref[def:V and E]{E}(0)$ to be the \textbf{root edge} of $(M,G)$, and $o:=\hyperref[def:V and E]{V}(0)$ to be \textbf{root vertex} (See the brown point in Figure \ref{fig:H-C bijection}). \label{def:V and E}
		
		\begin{figure}[htbp]
			\centering
			\vspace{0.1in}        	
			\includegraphics[scale=0.18]{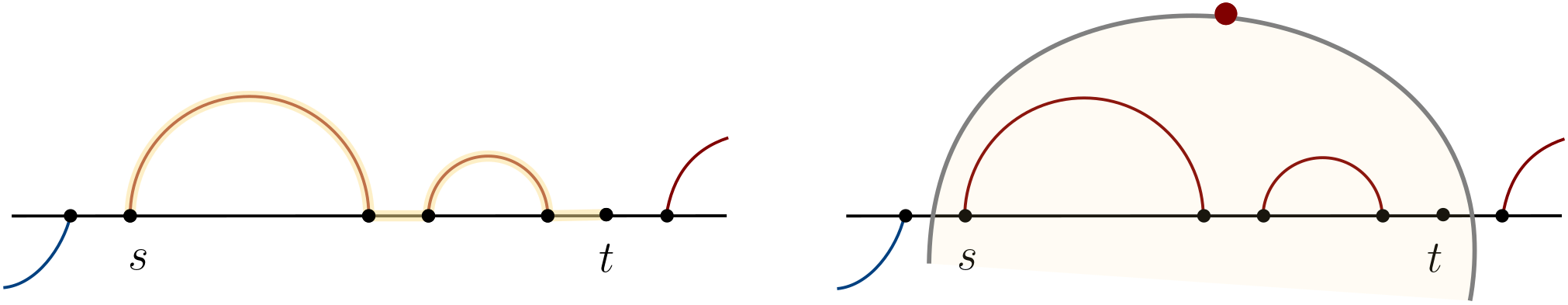}
			\caption{Region encoded by a reducible subword} 
			\label{fig:reducible}
		\end{figure}
		
		\vspace{0.1in}
		\noindent\textbf{Reducible word.} When $X(s,t)$ is reducible, $\hyperref[def:Tri]{\mathrm{Tri}}[s,t]$ forms a bubble-like region in the plane, enclosed by the two Tutte edges $\hyperref[def:V and E]{E}(s)$ and $\hyperref[def:V and E]{E}(t+1)$ (See the outermost two edges in Top Right of Figure \ref{fig:H-C bijection} for an example). The proof is as follows: we first look at the submap $\subset \hyperref[arc map]{\mathbb{A}}$ formed by (1) arcs with endpoints in $[s,t]$ and (2) line segments in the real axis between $[s,t]$. Consider the outermost boundary of this submap in the upper half-plane (see the yellow line on the left of Figure \ref{fig:reducible}). This line, in conjunction with the segments $[s-1,s]$ and $[t,t+1]$, will be boundary edges of a shared face $f$ in $\hyperref[arc map]{\mathbb{A}}$. This is because there are no arcs from $(-\infty,s-1]$ or $[t+1,\infty]$ to some vertices in $[s,t]$, all times in $[s,t]$ are matched to times also in $[s,t]$. See the left of Figure \ref{fig:reducible}. Thus, when taking the dual map $\hyperref[def:Tutte triangulation]{\mathbb{T}}$, both of the two Tutte edges $\hyperref[def:V and E]{E}(s)$ and $\hyperref[def:V and E]{E}(t+1)$ will connect to a common vertex $f^*$ in $\hyperref[def:Tutte triangulation]{\mathbb{T}}$, which is the associated vertex of face $f$. See the right of Figure \ref{fig:reducible}. For a similar reason, $\hyperref[def:V and E]{E}(s)$ and $\hyperref[def:V and E]{E}(t+1)$ will also connect to a common vertex in $\hyperref[def:Tutte triangulation]{\mathbb{T}}$ which is below the real axis.

		Recall in the \hyperref[Step3 of HC bijection]{\textit{final step}} of constructing $(M,G)$, we flip all blue or red edge that corresponds to an $\texttt{F}$ to its dual edge. We have the following:
		\begin{lemma}\label{lem:F-triangle}
			Each pair of endpoints from any edge in $G$ (resp. $G^*$) will always stay in the same blue (resp. red) cluster, regardless of any flip operation performed during the construction process.
		\end{lemma}
        
		\begin{proof}
			If this edge isn't adjacent to an \texttt{F}-triangle, then this edge also exists in $\hyperref[def:Tutte triangulation]{\mathbb{T}}$, thus the endpoints are always connected. 

            \begin{figure}[htbp]
    			\centering
    			\vspace{0.1in}
    			\includegraphics[scale=0.16]{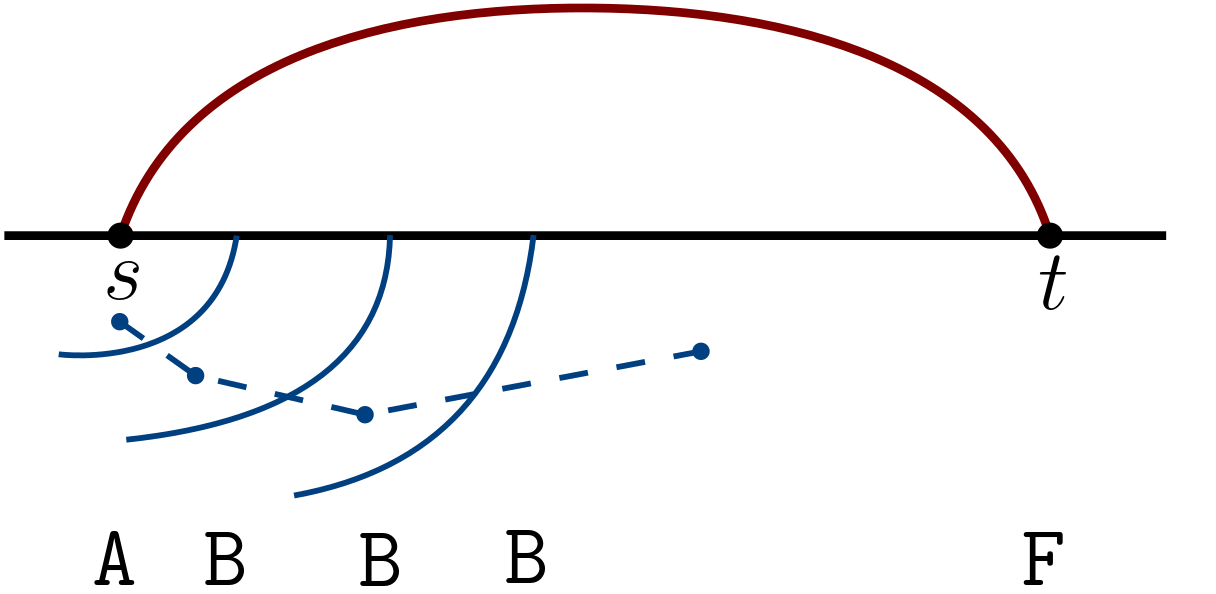}
    			\label{fig:F triangle}
		  \end{figure}
   
            For an $\texttt{F}$-triangle $\hyperref[def:Tri]{\mathrm{Tri}}(t)$, by symmetry, we can assume it is matched to a hamburger order $X(s)$. We claim the blue endpoints $\hyperref[def:V and E]{V^*}(s)$ and $\hyperref[def:V and E]{V^*}(t)$ should always stay in the same blue cluster when we flip the diagonals in the \hyperref[Step3 of HC bijection]{\textit{final step}} of the hamburger-cheeseburger bijection. Consider the blue arcs in arc map \hyperref[arc map]{$\hyperref[arc map]{\mathbb{A}}$}, since $X(s)$ is the freshest burger that is unconsumed before time $t$, all the blue arcs could not cross vertical line $x=t-1$. Then all the blue arcs that separate $s$ and $t$ will be some nested arcs enclosed vertex $s$ in \hyperref[arc map]{$\mathbb{A}$}. Moreover, these arcs could not have an endpoint that corresponds to an $\texttt{F}$, otherwise, this $\texttt{F}$-order should consume the hamburger in $s$ first. 
            
            Therefore, after taking the dual, we get a blue path from $\hyperref[def:V and E]{V^*}(s)$ to $\hyperref[def:V and E]{V^*}(t)$ in $\hyperref[def:Tutte triangulation]{\mathbb{T}}$, and these edges will not change during the flip operations.
		\end{proof}
		
		\noindent\textbf{\texttt{F} and loops.} 
		The $\texttt{F}$ in the word $X$ are in one-to-one correspondence with the FK loops in the map. For an $\texttt{F}$-triangle $\hyperref[def:Tri]{\mathrm{Tri}}(t)$, by symmetry, we can assume it is matched to a hamburger order $X(s)$. From the above paragraph, we see a blue path connecting $\hyperref[def:V and E]{V^*}(s)$ and $\hyperref[def:V and E]{V^*}(t)$. Thus after the flip operation, we have a blue edge connecting $\hyperref[def:V and E]{V^*}(s)$ and $\hyperref[def:V and E]{V^*}(t)$ inside the quadrangle, thus we get a blue loop that separates $\hyperref[def:V and E]{V}(t)$ and $\hyperref[def:V and E]{V}(s)$. See Figure \ref{fig:F and loops} for an illustration. Before the flip operation, $\hyperref[def:V and E]{V}(t)$ and $\hyperref[def:V and E]{V}(t+1)$ are in the same red cluster, and after the operation, they are separated by the blue loop. The remaining part of the red maps is unchanged. Thus we can see the number of red clusters is increased by $1$. The number of blue clusters is unchanged after the flip since $\hyperref[def:V and E]{V^*}(s)$ and $\hyperref[def:V and E]{V^*}(t)$ are already in the same cluster before. Now as the interfaces of blue and red clusters, the number of FK loops is increased by $1$.
		
		\begin{figure}[htbp]
			\centering
			\vspace{0.1in}
			\includegraphics[scale=0.17]{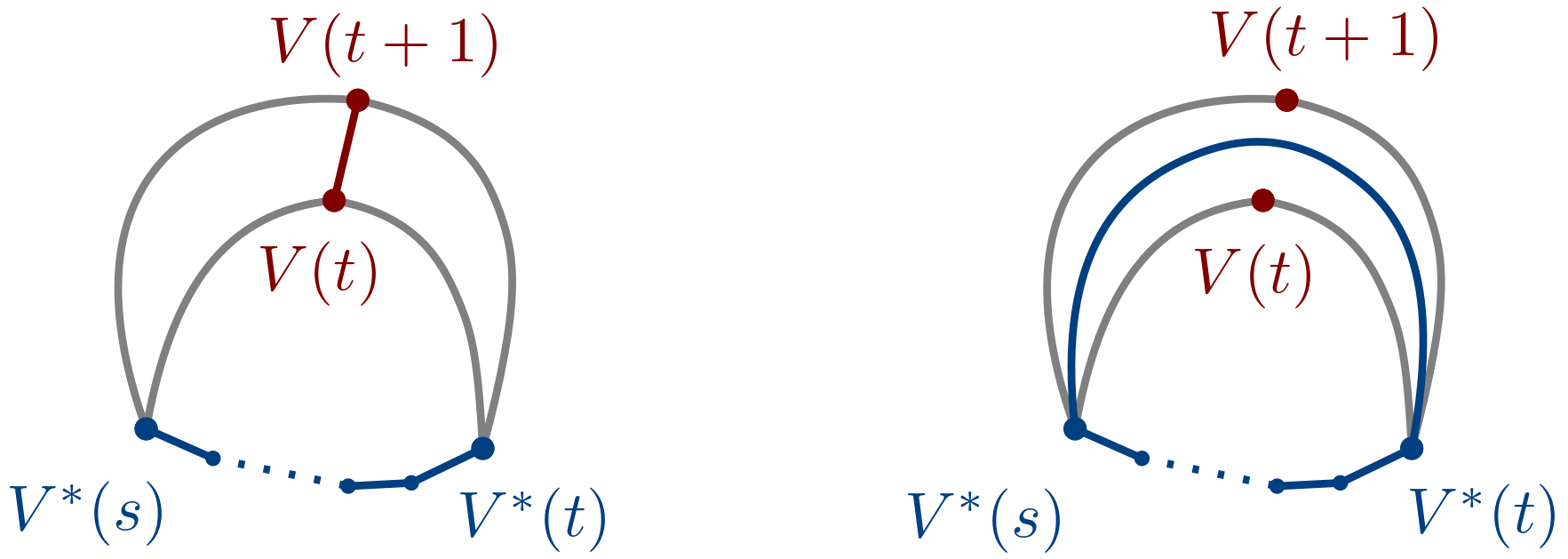}
			\caption{An \texttt{F} corresponds to a loop}
			\label{fig:F and loops}
		\end{figure}
		
		The reader can check the construction described here is consistent with the finite case described in \cite{she16}, where the bijection is proved. In that paper, it is also explained that the decorated, rooted, infinite random planar map $(\M,\G, o)$ corresponding to the random bi-infinite word of law $\P_p$ is an infinite $q$-FK map, with 
		\[
		\sqrt{q}= \frac{2p}{1-p} .
		\]
		In fact, let random letter $X(i)$ of law $\theta_p$, let $\mathbb{P}_p^{(n)}$ be the law of random word $w=(X_1...X_{2n}$ conditioned on the event that it is reducible. Let $w$ be a word of length $n$ which encodes $(\boldsymbol{m}_{n},\boldsymbol{g}_{n})$ via the finite-volume version of the above bijection. Then:
		$$\mathbb{P}_p^{(n)}(w)=\mathbb{P}_p^{(n)}(M_{n}=\boldsymbol{m}_{n}, G_{n}=\boldsymbol{g}_{n}) \propto\left(\frac{1}{4}\right)^{n}\left(\frac{1-p}{4}\right)^{\# \texttt{A}+\# \texttt{B}}\left(\frac{p}{2}\right)^{\# \texttt{F}}
		\propto \sqrt{q}^{\ell\left(\boldsymbol{m}_{n}, \boldsymbol{g}_{n}\right)}$$
		where $\# \texttt{A}$, $\# \texttt{B}$, $\# \texttt{F}$ are the numbers of respective letters in $w$, and $q\in[0,\infty)$ s.t. $2p/(1-p)=\sqrt{q}$. This gives exactly the law of $q$-FK planar map $P_{n,q}$.
		
		Let $(H_n)_{n \in \mathbb{Z}}$ be the net amount of hamburgers started at the origin. That is, we set $H_0=0$, replace all $\texttt{F}$ with the corresponding \texttt{A} or \texttt{B}, then let
		\[
		H_n= ( \# \texttt{a}-\# \texttt{A} \text{ in } [0,n-1] ) \text{ if } n>0,  \quad H_n=  (\# \texttt{A}-\# \texttt{a}\text{ in } [n,-1] ) \text{ if }n<0.
		\]
		Similarly, define the net amount sequence $(C_n)_{n \in \mathbb{Z}}$ of cheeseburgers. They are the contour functions of $\T$ (colored red in Figure \ref{fig:H-C bijection}) and its dual tree (colored blue in Figure \ref{fig:H-C bijection}), except there are some constant steps. We extend $H$ and $C$ to functions from $\mathbb R$ to $\mathbb R$ by linear interpolation.
		
		\begin{theorem}[\cite{she16},Theorem 2.5]\label{conv_to_BM}
			Let $L_t$, $R_t$ be correlated two-sided Brownian motions starting from 0, with covariance matrix
			\[
			\operatorname{Var}\left(L_{t}\right)=\operatorname{Var}\left(R_{t}\right)=\frac{1+\alpha}{4}|t|   ; \quad \operatorname{Cov}\left(L_{t}, R_{t}\right)=\frac{1-\alpha}{4}|t|
			\]
			where $\alpha=\max (1-2p, 0)$. Then, in distribution,
			\[
			(\varepsilon H_{t/\varepsilon^2}, \varepsilon C_{t/\varepsilon^2}), \rightarrow\left(L_{t}, R_{t}\right)
			\text{ as $\varepsilon \to 0$}
			\]
			with respect to the topology of uniform convergence on compact intervals.
		\end{theorem}
		
		From the theorem above, we can easily see why $p=1/2 \ (q=4)$ is a transition point. When $p>1/2$, $\alpha = 0$ and so $H_t$ and $C_t$ converge to the same two-sided Brownian motion.
		In the following chapters, we only consider the FK map in the $p>1/2$ ($q>4$) case.
		
		\begin{remark}
			Theorem~\ref{conv_to_BM} can be interpreted as a scaling limit statement for infinite FK planar maps in the so-called ``peanosphere sense". See~\cite{DMS14} \cite{GHS_survey} for a discussion of this type of convergence. Our goal is to upgrade this convergence to GHP convergence when $q>4$.
		\end{remark}

	\section{Spanning tree converges to CRT}\label{Chap_Tree_to_CRT}

	\subsection{Local GHP distance}
		We first recall some notations in metric geometry. The reader can also see \cite{local_GHP} for more information. Let $Z$ be a complete and separable metric space endowed with a metric $d$. The \textbf{Hausdorff distance} of two Borel set $A$ and $B$ is given by:
		\[
		d_H(A,B):=\inf\{\varepsilon >0 \mid A \subset B^\varepsilon, B \subset A^\varepsilon\}
		\]
		where for a set $S$, we define $S^\varepsilon :=\{x \in Z \mid d(x,S) < \varepsilon \}$. 
		
		For two finite measures $\mu$ and $\sigma$ on $Z$, the \textbf{Prokhorov distance} between them is defined by:
		$$d_P(\mu,\sigma):= \inf\{\varepsilon >0 \mid 
		\mu(A) \leq \sigma(A^\varepsilon)+\varepsilon, 
		\sigma(A) \leq \mu(A^\varepsilon)+\varepsilon, \text{for any Borel set } A \}.$$
		
		Now we define the Gromov-Hausdorff-Prokhorov (GHP) distance between pointed spaces $(X,d_X,\mu_X,o_X)$ (we will simply say $X$ if no ambiguity), where $X$ is a compact topological space, endowed with a metric $d_X$, a finite measure $\mu_X$, and a root $o_X$. For $(X,d_X,\mu_X,o_X)$ and $(Y,d_Y,\mu_Y,o_Y)$,
		\[
		d_\mathrm{GHP}(X,Y):=\inf_{\phi_X,\phi_Y, Z}
		\{d^Z(\phi_X(o_X),\phi_Y(o_Y)) +
		d^Z_H(\phi_X(X),\phi_Y(Y)) +
		d^Z_P(\phi_{X,*}(\mu_X),\phi_{Y,*}(\mu_Y)\}
		\]
		where $(\phi_X, \phi_Y, Z)$ runs over all isometric embeddings $\phi_X: X \hookrightarrow$ $(Z,d^Z)$ and $\phi_Y: Y \hookrightarrow$ $(Z,d^Z)$. $\phi_{X,*}(\mu_X)$ and $\phi_{Y,*}(\mu_Y)$ are the pushforward measures. Let $\mathbb{K}$ be the set of all such spaces $(X,d_X,\mu_X,o_X)$, viewed modulo isometries which preserve the measure and the marked point. Then ($\mathbb{K},d_\mathrm{GHP}$) is a complete separable metric space \cite{local_GHP}. 
		
		\begin{definition}(Distortion)
			A \textbf{correspondence} $\mathcal{R}$ of two pointed compact metric spaces $(X,d_X,o_X)$ and $(Y,d_Y,o_Y)$ is a subset of $X \times Y$ such that every element of $X$ and every element of $Y$ appears in at least one pair in $\mathcal{R}$ and $(o_X,o_Y) \in \mathcal{R}$. We define the \textbf{distortion} of $\mathcal{R}$ by:
			\[
			\operatorname{dis}(\mathcal{R})=\sup\{ |d_X(x_1,x_2)-d_Y(y_1,y_2)|: (x_1,y_1), (x_2,y_2) \in \mathcal{R}\} .
			\]
		\end{definition}
		
		We say a measure $\nu$ is a \textbf{coupling} of $\mu_X$ and $\mu_Y$ if it is a measure on $X \times Y$ with marginal laws $\mu_X$ and $\mu_Y$. The following equivalent definition of GHP distance will be helpful in later chapters. See e.g. \cite[Proposition 6]{Tessellations} for proof.
		\begin{proposition}[\cite{Tessellations}]\label{distortion}
			For compact metric spaces with finite measures $(X,d_X,\mu_X)$ and $(Y,d_Y,\mu_Y)$:
			$$d_\mathrm{GHP}(X,Y)=\inf_{ \mathcal{R},\nu} \ \frac{1}{2}\operatorname{dis}(\mathcal{R}) + \nu(\mathcal{R}^c)$$
			where $(\mathcal{R}, \nu)$ runs through all pairs consisting of a correspondence of $X$ and $Y$ and a coupling of $\mu_X$ and $\mu_Y$.
		\end{proposition}
		
		Since the protagonist in our paper is not a compact space, we have to extend the definition of the GHP distance. We first recall the definition of length space. In $(X,d)$, we define the length of curve $P:[a,b] \to X$ by:
		$$\mathrm{Len}(P):=\sup _{t_0=a<...<t_n=b}\sum d\left(P\left(t_i\right), P\left(t_{i-1}\right)\right).$$
		We say $(X,d)$ is a length space, if for any $x,y \in X$,
		$$d(x,y)=\inf_{P:x\to y}\mathrm{Len}(P).$$

		We consider complete locally compact length spaces $(X,d,\mu,o)$ endowed with a locally finite measure $\mu$ and a marked point (root) $o$. 
		
		\begin{definition}\label{notation:Br}
			We use $X^r$ to represent $r$-ball around the root, i.e. $X^r:=\{x \in X \mid d(x,o) \leq r\}$. 
			$\mu^r$ is the restriction of $\mu$ to $X^r$, i.e. $\mu^r(A)=\mu(A\cap X^r)$. $d^r:=d|_{X^r \times X^r}$ where we regard $d$ as a function $X \times X \to \mathbb{R}_{\geq 0}$. We will use $X^r$ to represent $(X^r,d^r,\mu^r,o)$ if there is no ambiguity.
		\end{definition} 
		
		Since $X^r$ is a compact space with finite measure, we can define the local GHP distance between $X$ and $Y$ by:
		\[
		d_\mathrm{GHP}^\mathrm{loc}(X,Y)=\int_0^\infty e^{-r}(1 \wedge d_\mathrm{GHP}(X^r,Y^r)) dr .
		\]
		Let $\mathbb{L}$ be the space of all complete locally compact length space endowed with a locally finite measure, viewed modulo isometries which preserve the measure and the marked point. $\mathbb{L}$ is checked in \cite{local_GHP} to be a complete separable metric space.
		
		In the following context, we will consider the local GHP distance between the infinite CRT and infinite FK map $M$. It will be easy to see that the infinite CRT belongs to $\mathbb{L}$. To make the planar map $M$ into an element of $\mathbb{L}$, we extend its graph metric to the points inside the edges in an isometric way. To be more specific, for each edge $e \in M$, we assign a homeomorphism $f_e:[0,1] \to e$ then define $d(x,y) := |f_e^{-1}(x) - f_e^{-1}(y)|$. For the distance between $x$ and $y$ in different edges, we set $$d(x,y):=\inf \sum_{i=0}^{n-1} d_M(x_i,x_{i+1})$$
		where the infimum runs over all $n \in \mathbb{Z}_+{\geq 0}, x_0=x,x_n=y,x_i(0<i<n)$ are vertices of $M$. Such an extended map is often called a \textbf{metric graph}. The readers can easily check this metric graph is indeed a complete length space and the new metric is consistent with $d_M$ when restricted to the vertices of $M$. The metric graph of the spanning tree $T \subset M$ is defined as a sub metric graph, that is, $d_T(x,y)$ is the infimum of lengths of paths that only go through edges in $T$. Note each point has a distance smaller than $1$ from its nearest vertex, thus after scaling, it will have no effect on the convergence result. Thus for brevity, when comparing local GHP distance between metric spaces, we will only consider the graph distance between original vertices.

	\subsection{Infinite CRT}\label{chap_crt}
		
		We will now define the infinite version of the CRT, encoded by two-sided Brownian motion started at $0$ from the origin. For a continuous function $g: \mathbb{R} \rightarrow \mathbb{R}$, $g(0)=0$, when $s<t$, let $m_g(s,t)=\inf_{x\in[s,t]} g(x)$. Define $d_g(t,s)=g(s)+g(t)-2m_g(s,t)$ and let $d_g(t,s)=d_g(s,t)$. $d_g$ can be easily checked to be a pseudo metric. We say $s \sim t$ if $d_g(s,t)=0$. Note that this is an equivalence relation. 
		
		Let \textbf{$T_g =\mathbb{R} / \sim$} and let \textbf{$\proj_g$}\label{def:proj} be the projection map $\mathbb{R} \to T_g$. We give $T_g$ the metric induced by $d_g$, the root $\proj_g(0)$, the measure $\lambda_g$ as the pushforward of Lebesgue measure $\lambda$ in $\mathbb{R}$. We call $T_g$ the (infinite) tree encoded by the function $g$. It is easy to check this infinite tree is locally finite. Thus we can talk about the local GHP distance between infinite trees of this type.

		The two-sided Brownian motion $L=(L_t)_{t\in\mathbb R}$ in Theorem \ref{conv_to_BM} is a.s.\ continuous. We can use it to encode an infinite random tree as above, and the resulting quotient space is called the infinite CRT, denoted by $\CRT$\footnote{Since $L$ in Theorem \ref{conv_to_BM} is equal to $1/2$ times standard Brownian motion, $\CRT$ is equal in law to the ordinary CRT with distances scaled by $1/2$}. Let $\dCRT(\proj_L(s),\proj_L(t)):=L_s+L_t-2m_L(s,t)$ denote its metric, $\lebCRT$ denote its measure (the pushforward of Lebesgue measure on $\mathbb R$ under the quotient map), and $\oCRT :=\proj_L(0)$ denote its root. \label{def:CRT}
		
		Recall that the function $(H_t)_{t\in \mathbb R}$ as in Theorem \ref{conv_to_BM} is extended from $\mathbb Z$ to $\mathbb R$ by piecewise linear interpolation, so is continuous. The quotient space $T_H$ is our canonical spanning tree $\T$ (viewed as a metric graph). The reader can check easily that $d_H$ defined above is equal to the (metric) graph distance $\dT$ on $T$. Let $\lebT$ be the pushforward of the Lebesgue measure on $T$. Inspired by the convergence of uniform trees to the CRT\cite{CRT}, we expect the following theorem.
		
		\begin{thm}\label{tree_to_CRT}
			$(\T,\dT/n,\lebT/n^2,o) \rightarrow (\CRT,\dCRT,\lebCRT,\oCRT)$ in distribution, w.r.t local GHP distance. 
		\end{thm}

		The rest of this section is devoted to the proof of Theorem~\ref{tree_to_CRT}. Readers familiar with classical results, such as those presented in \cite{legall}, may question why it is not immediately apparent that the convergence of contour functions implies the convergence of trees. Although the proofs are consistent in spirit, the scenario we are considering here is an infinite version, which not only implies the presence of more intricate details but also suggests that some theorems applicable in finite cases may not hold true. For instance, when we truncate an excursion containing zero from an infinite Brownian Motion, one can easily see that it is not a Brownian excursion (we have indeed fixed the height of the excursion, a condition that does not apply to a Brownian excursion). This implies that when a subtree is truncated near the root of an infinite CRT, its law is not simply the law of a finite CRT. What's more, the function $(H_t)_{t \in \mathbb{R}}$ we consider here is not exactly the contour function of the tree, since it has some constant steps.
		
		Recall from Theorem \ref{conv_to_BM} that $(\frac{1}{n}H_{n^2t})_{t \in \mathbb{R}}\rightarrow (L_t)_{t \in \mathbb{R}}$ in distribution w.r.t. local uniform metric. The space of two-sided continuous functions is separable, by the Skorokhod theorem, there exists a sequence of random functions $(Y_t^{(n)})_{t\in \mathbb{R}}\eqd (\frac{1}{n}H_{n^2t})_{t\in \mathbb{R}}$ $(n \in \mathbb{Z_+})$ with a coupling of $(L_t)$, such that $(Y^{(n)}_t) \rightarrow (L_t)$ uniformly in any compact interval as $n \rightarrow \infty$. Let $\scale{T}\eqd (\T,d_T/n,\lebT/n^2,o)$ be the tree encoded by $\scale{Y}$. For simplicity, let $\scaleproj:=\proj_{\scale{Y}}$ \label{def:scaleproj}be the be the quotient map $\mathbb{R} \to \scale{T}$, the root $\scaleproj(0)$ as $o_n$, the measure $\lambda_n:=\hyperref[def:scaleproj]{\mathcal{P}_{n*}}(\lambda)$ where $\lambda$ is the Lebesgue measure, and
		\[
		d_n(\scaleproj(x),\scaleproj(y)):=d_{\scale{Y}}(\scaleproj(x),\scaleproj(y)) \eqd n^{-1}d_T(\proj_H(n^2 x),\proj_H(n^2 y)).
		\]
		
		\begin{lemma}\label{uniform distance between T and CRT}
			For any $0<r<\infty$, a.s. $\sup_{x,y \in [-r,r]} |\dCRT(\proj_L(x),\proj_L(y))-d_n(\scaleproj(x),\scaleproj(y))| \rightarrow 0$ as $n \rightarrow \infty$.
		\end{lemma}
		\begin{proof}
			By definition, $\dCRT(\proj_L(x),\proj_L(y))=L(x)+L(y)-2m_L(x,y)$ and $d_n(\scaleproj(x),\scaleproj(y))=Y^{(n)}(x)+Y^{(n)}(y)-2m_{Y^{(n)}}(x,y).$
			Thus, since $Y^{(n)}$ converges locally uniformly to $L$,
			$$\sup_{x,y \in [-r,r]} |\dCRT(\proj_L(x),\proj_L(y))-d_n(\scaleproj(x),\scaleproj(y))| \leq 4 \sup_{x \in [-r,r]} |L(x)-\scale{Y}(x)| \rightarrow 0.$$
		\end{proof}
		
		We write $B_r(\scale{T}):=\{x \in \scale{T} \mid d_n(x,o_n) \leq r \}$ for the $r$-ball around the root vertex of $\scale{T}$ and $B_r(\CRT)$ for the $r$-ball around the root vertex of $\CRT$. Let's define
		\begin{align*}
			&\lt_r = \sup\{t<0 \mid L_t<-r\} ,&&\rt_r = \inf\{t>0 \mid L_t<-r\} ,\\
			&\lt_{r}^{(n)} = \sup\{t<0 \mid \scale{Y}_t<-r\} ,  &&\rt_{r}^{(n)} = \inf\{t>0 \mid \scale{Y}_t<-r \}.
		\end{align*}
		We then have $I_r \subset [s_r,t_r]$ and $I_{n,r} \subset [\scale{s_r},\scale{t_r}]$, where 
		\begin{equation} \label{eqn:def-intervals}
			I_r:=\proj_L^{-1}(B_r(\CRT)) \quad \text{and} \quad I_{n,r}:=\scaleproj^{-1}(B_r(\scale{T})) . 
		\end{equation}
		
		\begin{lemma} \label{lem:inf-times}
			Almost surely, for any $\delta>0$, when $n$ is large enough, $[\lt_{r}^{(n)},\rt_{r}^{(n)}] \subset [\lt_r-\delta,\rt_r+\delta]$.
		\end{lemma}
		
		\begin{proof}
			By a property of Brownian motion, $[\lt_r-\delta,\rt_r+\delta]$ is a.s.\ a compact interval. Thus $\scale{Y}_t \rightarrow L_t$ uniformly in this internal. Consider any subsequential limit $\tau=\lim_{i \rightarrow \infty} \rt_{r}^{(n_i)} \in [0,\infty]$. We know a.s. $\exists u \in [\rt_r,\rt_r+\delta)$ such that $L_{u} < -r$. We have $\lim_{i \rightarrow \infty} Y_{u}^{(n_i)}=L_{u} < -r$, thus a.s. for large enough $n_i$, $\rt_{r}^{(n_i)}\leq u < \rt_r+\delta$ $\Rightarrow \tau \leq \rt_r+\delta$. By symmetry, the same is true for $\lt_{r}^{(n)}$.
		\end{proof}
		
		\begin{lemma}\label{compact T ball}
			Fix $r>0$, a.s. $\exists R<\infty$ which only depends on $L$ and $r$, such that $I_{n,r} \subset [-R,R]$ for large enough $n$.
		\end{lemma}
		
		\begin{proof}
			By the above proposition, when $n$ large enough, $ [\scale{\lt}_{r},\scale{\rt}_{r}] \subset [\lt_{r}-\delta,\rt_{r}+\delta]$, which is a.s. compact so contained in $[-R,R]$ for some $R$ $\Rightarrow I_{n,r} \subset [\lt_{r}^{(n)},\rt_{r}^{(n)}] \subset [-R,R]$.
		\end{proof}

		\begin{proposition}\label{compare CRT ball and T ball}
			Fix $r>0$. $\forall \varepsilon>0$, a.s. when $n$ large enough, $I_r \subset I_{n,r+\varepsilon}$ and $I_{n,r} \subset I_{r+\varepsilon}$.
		\end{proposition}
		
		\begin{proof}
			We know $B_r(\CRT)$ is a.s. compact. By this and Lemma \ref{compact T ball}, a.s.\ there exists $R>0$ such that $I_r \cup I_{n,r} \subset [-R,R]$ for large enough $n$. By Lemma \ref{uniform distance between T and CRT},
			$$\sup_{x \in [-R,R]} |\dCRT(\oCRT,{\proj_L}(x))-d_n(o_n,\scaleproj(x))| \rightarrow 0.$$
			Thus $ \forall \varepsilon >0$, when $n$ is large enough, this amount will be controlled by $\varepsilon$. Thus $x \in I_{n,r}$ implies $\dCRT(\oCRT,{\proj_L}(x)) \leq r+\varepsilon$. Hence $x \in I_{r+\varepsilon}$. Similarly $I_r \subset I_{n,r+\varepsilon}$.
		\end{proof}
		
		To estimate the local GHP distance, we first compare the GHP distance between balls. Recall notations in Definition \ref{notation:Br}: Let $\lambda$ be Lebesgue measure. For $r>0$, we endow $B_r(\scale{T})$ with restricted measure $\lambda_n^r(\cdot):=\lambda(\scaleproj^{-1}(\cdot \cap B_r(\scale{T}))=\lambda(\scaleproj^{-1}(\cdot) \cap I_{n,r})$, and metric $d_n^r$ as restriction of metric $d_n$ in set $B_r(\scale{T}) \times B_r(\scale{T})$. Similarly, we endow $B_r(\CRT)$ with restricted measure $\lebCRT^r(\cdot):=\lambda(\proj_L^{-1}(\cdot) \cap I_{r})$ and metric $\dCRT^r$ as restriction of metric $\dCRT$ in set $B_r(\CRT) \times B_r(\CRT)$.
		
		\begin{proposition}\label{discrete tree}
			For any fixed $r>0$, as $n\rightarrow \infty$, $d_{\mathrm{GHP}}(B_r(\scale{T}),B_r(\CRT)) \to 0$ a.s.
		\end{proposition}
		
		\begin{proof}
			Let $I_r$ and $I_{n,r}$ be as in~\eqref{eqn:def-intervals} and let $J_n=I_r \cap I_{n,r}$. We first show that as $n \to \infty$,
			\begin{equation}\label{claim_in_GHPdistance_betweeen_balls}
				d_\mathrm{GHP}\left((B_r(\scale{T}),d_n,\hyperref[def:scaleproj]{\mathcal{P}_{n*}}(\lambda \vert_{J_n}),o_n),\ \left(B_r(\CRT),\dCRT,\proj_{L*}(\lambda \vert_{J_n}),\oCRT\right)\right) \to 0.
			\end{equation}
			
			To use Lemma \ref{distortion}, we define a correspondence $\mathcal{R}_{n,r}$ between $B_r(\scale{T})$ and $B_r(\CRT)$:
			
			\begin{enumerate}[$\quad$ 1)]
				\item For $t \in I_r \cap I_{n,r}$, let $\proj_L(t)$ correspond to $\scaleproj(t)$.
				\item For $t \in I_r \setminus I_{n,r}$, let $\proj_L(t)$ correspond to $x \in B_r(\scale{T})$ such that $d_n(x,\scaleproj(t))$ is minimal. 
				\item For $t \in I_{n,r} \setminus I_r$, let $\scaleproj(t)$ correspond to $x \in B_r(\CRT)$ such that $\dCRT(x,\proj_L(t))$ is minimal. 
			\end{enumerate}   
			
			For any $\varepsilon>0$, by Lemma \ref{compare CRT ball and T ball}, we can let $n$ large enough such that $I_r \subset I_{n,r+\varepsilon}$ and $I_{n,r} \subset I_{r+\varepsilon}$. Thus via triangle inequality, we have
			\begin{align*}\label{equation_distortion}
				\operatorname{dis}(\mathcal{R}_{n,r}) \leq \sup_{x,y\in I_r \cup I_{n,r}} |\dCRT(\proj_L(x),\proj_L(y))-d_n(\scaleproj(x),\scaleproj(y))|&
                \\+\sup_{t\in I_r \setminus I_{n,r}}d_n(B_r(\scale{T}),\scaleproj(t)) &+\sup_{t\in I_{n,r} \setminus I_r}\dCRT(B_r(\CRT),\proj_L(t)).
			\end{align*}
			By Lemma \ref{uniform distance between T and CRT} and Lemma \ref{compact T ball}, the first item on the right side converges to zero. Since $I_r \subset I_{n,r+\varepsilon}$, we have for any $t \in I_r$, $d_n(B_r(\scale{T}),\scaleproj(t))\leq \varepsilon$. Similarly, $I_{n,r} \subset I_{r+\varepsilon} \Rightarrow \dCRT(B_r(\CRT),\proj_L(t)) \leq \varepsilon$ for any $t \in I_{n,r}$. Let $\varepsilon \to 0$, we have $\operatorname{dis}(\mathcal{R}_{n,r}) \to 0$ as $n \to \infty$.
			
			To compare the measures, we define a coupling measure $\nu_n$ of $\hyperref[def:scaleproj]{\mathcal{P}_{n*}}(\lambda \vert_{J_n})$ and $\proj_{L*}(\lambda \vert_{J_n})$ such that $\nu_n(A \times B)=\lambda( \proj_L^{-1}(A)\cap \scaleproj^{-1}(B) \cap J_n)$. One can easily check the existence of $\nu_n$ and it is with marginal laws $\hyperref[def:scaleproj]{\mathcal{P}_{n*}}(\lambda \vert_{J_n})$ and $\proj_{L*}(\lambda \vert_{J_n})$. Note $\nu_n(\mathcal{R}_{n,r}) \geq \nu_n(\proj_L(J_n) \times \proj_L(J_n))= \lambda (J_n)$, we have $\nu_n(\mathcal{R}_{n,r}^c)=0$. Now by Lemma \ref{distortion}, as $n \to \infty$, LHS of (\ref{claim_in_GHPdistance_betweeen_balls}) $\leq \frac{1}{2}\operatorname{dis}(\mathcal{R}_{n,r}) + \nu(\mathcal{R}_{n,r}^c)=\frac{1}{2}\operatorname{dis}(\mathcal{R}_{n,r}) \to 0.$
			
			Then we continue to prove as $n \to \infty$,
			\begin{equation}\label{claim_2}			
            d_\mathrm{GHP}\left((B_r(\scale{T}),d_n,\hyperref[def:scaleproj]{\mathcal{P}_{n*}}(\lambda \vert_{J_n}),o),\ (B_r(\scale{T}),d_n,\hyperref[def:scaleproj]{\mathcal{P}_{n*}}(\lambda \vert_{I_{n,r}}),o)\right) \to 0.
			\end{equation}
			Note they have the same metrics, so the GHP distance is just the Prokhorov distance of measures. For any $\varepsilon >0$, when $n$ is large enough, $I_{n,r} \subset I_{r+\varepsilon}$. We have:
			$$d_\mathrm{P}\left(\hyperref[def:scaleproj]{\mathcal{P}_{n*}}(\lambda \vert_{J_n}),\hyperref[def:scaleproj]{\mathcal{P}_{n*}}(\lambda \vert_{I_{n,r}})\right) \leq \lambda(I_{n,r} \setminus I_{r}) \leq \lambda(I_{r+\varepsilon}\setminus I_r).$$
			We claim that for any fixed $r > 0$, we have $\lambda(I_{r+\delta}\setminus I_r) \rightarrow 0$ as $\delta \rightarrow 0$. In fact, this is equivalent to the statement that the Lebesgue measure of set $\{t \in \mathbb{R} \mid \dCRT(o,\proj_L(t))=r\}$ is zero. Since $\dCRT(o,\proj_L(t))=L(t)-2\inf_{u\in[0,t]}L(u)$, the claim can be inferred easily from basic properties of Brownian motion. Letting $\varepsilon \to 0$, we prove (\ref{claim_2}). 
			
			We can similarly prove as $n \to \infty$,
			\begin{equation}\label{claim_3}
				d_\mathrm{GHP}\left((B_r(\CRT),\dCRT,\proj_{L*}(\lambda \vert_{J_n}),\oCRT),\ (B_r(\CRT),\dCRT,\proj_{L*}(\lambda \vert_{I_r}),\oCRT)\right) \to 0.
			\end{equation}
			Combine (\ref{claim_in_GHPdistance_betweeen_balls}), (\ref{claim_2}), (\ref{claim_3}), we prove the proposition.
		\end{proof}
		
		\begin{proof}[Proof of Theorem \ref{tree_to_CRT}]
			Note $e^{-r}(1 \wedge d_\mathrm{GHP}(B_r(\scale{T}) \, B_r(\CRT))) \leq e^{-r}$, and $e^{-r}$ is integrable. By the dominated convergence theorem, we have $$\lim_{n\rightarrow \infty} \int_0^{\infty} e^{-r}(1 \wedge d_\mathrm{GHP}(B_r(\scale{T}) \, B_r(\CRT))) dr=\int_0^{\infty} \lim_{n\rightarrow \infty} e^{-r}(1 \wedge d_\mathrm{GHP}(B_r(\scale{T}) \, B_r(\CRT))) dr=0.$$
		\end{proof}
		
		Consider a measure $\mu$ on $M$ which assigns each vertex a mass equal to its degree in $M$. In the subsequent chapters, we will give $T$ the measure $\mu=\mu \vert_T$. Recall the way we construct $T$ from the hamburger walk process $H$, we assert that $\mu \vert_T$ is exactly the pushforward measure of the counting measure $\mu_\mathbb{Z}$ on $\mathbb{Z}$.

        To provide a more intuitive explanation, we introduce another construction of the planar map $(M,T)$ known as the (discrete) mating of trees. First, put the burger processes $H$ (turn upside down) and $C$ in the plane as shown on the left of Figure \ref{fig:counting_measure}. Then draw disjoint grey lines connecting $H_t$ and $C_t$ for each $t\in \mathbb{Z}$. Subsequently, we glue points in $H$ (respectively, $C$) if they reside on the same level line and the segment of the level line between them does not cross $H$ (respectively, $C$). See an example on the left of Figure \ref{fig:counting_measure}, where the level of $H$ is drawn orange. 
        
        The reader can verify that this method of gluing continuous functions gives the same map as the method described at the start of this chapter. Moreover, we get a triangulation, as depicted on the right of Figure \ref{fig:counting_measure}, and it is exactly the Tutte triangulation of $(M,T)$, where the grey edges are exactly the Tutte edges. From this perspective, for a vertex $x \in M$, $\mu(x)=\operatorname{deg}_M(x)$ is not only equal to the count of Tutte edges with $x$ as the endpoint but also equal to the count of points $t \in \mathbb{Z}$ such that $\proj_H(t)=x$. 
		
		\begin{figure}[htbp]
			\centering
            \vspace{0.1in}
			\includegraphics[scale=0.22]{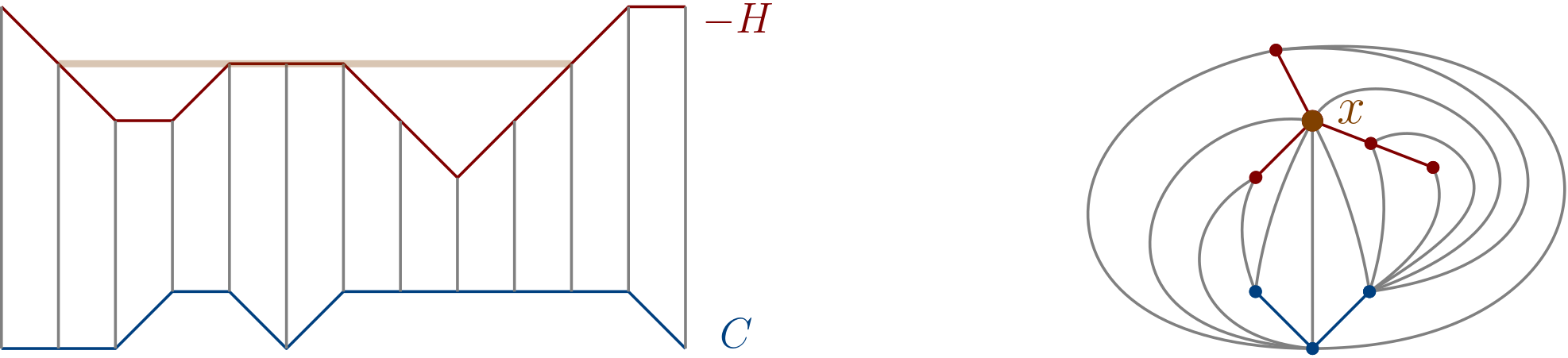}
			\caption{Discrete mating of trees}
			\label{fig:counting_measure}
		\end{figure}
        \vspace{0.1in}
        
        Now, let's compare Prokhorov distance of pushforward measures $\lebT:=\proj_{H*}(\lambda)$ and $\mu=\proj_{H*}(\mu_\mathbb{Z})$ in metric space $(T,\varepsilon\dT)$. For any Borel set $I \subset \mathbb{R}$, $I^\delta:=\{ x \mid \exists \ y \in I, \ |x-y| \leq \delta \}$. For any set $A \subset T$, $A^\delta:=\{ x \in T \mid \exists \ y \in A, \ \varepsilon\dT(x,y) \leq \delta \}$. Now, for any set $A$ in $T$, let $I=\proj_H^{-1}(A)$. By the definition of $\dT$ we know $\proj_H(I^1) \subset A^\varepsilon$. Thus,
        $$\lebT(A^\varepsilon) = \lambda \left(\proj_H^{-1}(A^\varepsilon)\right) \geq \lambda(I^1) \geq \sum_{x \in \mathbb{Z}, [x,x+1) \subset I^1} \lambda\left([x,x+1)\right) \geq \mu_{\mathbb{Z}}(I)=\mu(A),$$
        $$\mu(A^\varepsilon)= \mu_{\mathbb{Z}} (\proj_H^{-1}(A^\varepsilon)) \geq  \mu_{\mathbb{Z}}(I^1) \geq \# \{x \in \mathbb{Z} \mid \exists \ t \in [x,x+1), t \in I\} \geq \lambda(I)=\lebT(A).$$
		Now we get the Prokohorov distance of $\varepsilon^2\lebT$ and $\varepsilon^2\mu$ is smaller than $\varepsilon$. From the above discussion and Theorem~\ref{tree_to_CRT} we get the following theorem.
		\begin{theorem}\label{main_chap3}
			$(T, \varepsilon\dT, \varepsilon^2\mu,o) \rightarrow (\CRT,\dCRT,\lebCRT,\oCRT)$ in distribution w.r.t. local GHP distance.
		\end{theorem}
		
	\section{FK map converges to spanning tree}\label{Chap_map_to_tree}
	\subsection{Pinch point}\label{chap_4.1}
		Recall map $\hyperref[def:Tutte triangulation]{\mathbb{T}}$ that consists of edges in $G$ and $G^*$ and Tutte edges of $(M,G)$, is called the Tutte triangulation of $(M,G)$. It can be viewed as a rooted planar map, where the root edge is the same as $M$. From the hamburger-cheeseburger bijection, a reducible finite subword $X(s,t)$ will correspond to a bubble-like region in $\hyperref[def:Tutte triangulation]{\mathbb{T}}$, enclosed by Tutte edges $\hyperref[def:V and E]{E}(s)$ and $\hyperref[def:V and E]{E}(t+1)$ (see Figure \ref{fig:reducible}). Therefore, if one runs a path in $M$ that goes from vertices inside $\hyperref[def:Tri]{\mathrm{Tri}}[s,t]$ to the outside, one must go through vertex $\hyperref[def:V and E]{V}(s)=\hyperref[def:V and E]{V}(t+1)$.  
		
		For the $k$-th smallest reducible word $[s,t]$ that contains $u \in \mathbb{Z}$, if it exists, we call vertex $y:=\hyperref[def:V and E]{V}(s)=\hyperref[def:V and E]{V}(t+1)$ the \textbf{$k$-th pinch point} of $x:=\hyperref[def:V and E]{V}(u)$, and write $y=p_k(x)$. We also write $p(x)$ for the first pinch point and $p_0(x)$ for $x$ itself.
  
		The region in $\hyperref[def:Tutte triangulation]{\mathbb{T}}$ corresponding to the ($k$-th) reducible word after removing all nested reducible subwords is called the ($k$-th) \textbf{bubble}, and the regions corresponding to reducible subwords are called \textbf{holes}. A bubble together with its holes is called a filled bubble. See Figure \ref{pinch_point} for an illustration.

		\begin{figure}[htbp]
			\centering
			\vspace{0.2in}
			\includegraphics[scale=0.22]{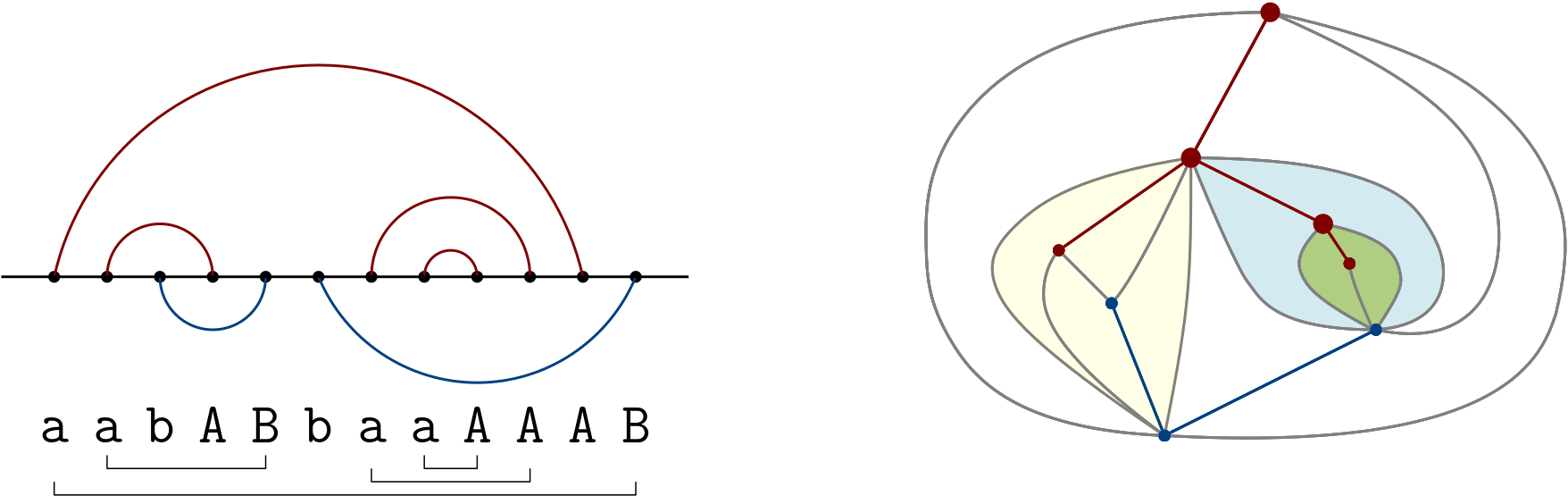}
			\caption{A reducible word with nested reducible subwords encoding a region with holes. The pinch points are the red points with larger sizes. The objects in the dual map are colored blue. The regions in different colors are the bubbles. Be cautious that while there is another blue point on a bubble's boundary, it isn't an exit of the bubble. This is because it's not a vertex in $M$, but rather a vertex in $M^*$.}
			\label{pinch_point}
		\end{figure}
			
		\begin{proposition}[Spatial Markov]\label{spatial_Markov}
			Let $B$ be the innermost filled bubble that contains $o$, and let $e_1$ and $e_2$ be the two Tutte edges that enclose $B$. Then collapse $B$, and glue $e_1$ and $e_2$ into single edge, the remaining map, re-rooted at this new edge, will have the same distribution as $\hyperref[def:Tutte triangulation]{\mathbb{T}}$, and independent of $B$.
		\end{proposition}
		
		\begin{proof}
			By Lemma \ref{lemma_locality}, whether or not $X(s,t)$ is reducible, only depends on $ (X(i))_{s\leq i \leq t}$. Since the symbols of $X$ are i.i.d., if $0\in [s,t]$, then when we condition on the event that $X(s,t)$ is reducible, the conditional law of $(...X(s-1)X(t+1)...)$ is the same as the law of $X$ (here index $t+1$ takes the place of $0$). Via the hamburger-cheeseburger bijection, we get the spatial Markov property for the map.
		\end{proof}

        We can also get the strong spatial Markov property easily, which will be used in Chapter \ref{chap_5}:

        \begin{corollary}\label{prop:strong_markov}
            Let $T$ be a finite stopping time. The proposition is still true when $B$ is the $T$-th innermost filled bubble that contains $o$.
        \end{corollary}
		
		\begin{proposition}[Translation invariance]\label{prop:translation_invariance}
			For fixed $t$, if we reroot the map from the root edge $\hyperref[def:V and E]{E}(0)$ to the edge $\hyperref[def:V and E]{E}(t)$, the new map will have the same distribution as the original map.
		\end{proposition}
		
		\begin{proof}
			Note the distribution of $X$ is invariant under shift operator $(X_i) \mapsto (X_{i+t})$, we know the law of the $\hyperref[def:Tutte triangulation]{\mathbb{T}}$ (thus also the law of FK map) is invariant under re-rooting, via hamburger-cheeseburger bijection. 
		\end{proof}
		
		\begin{lemma}[Decomposition of the path]\label{lem:Decomposition of the path}
			For fixed $t$, let $x=\hyperref[def:V and E]{V}(t)$ be a vertex in $M$. The geodesic $l$ (for either $\dT$ or $\dM$) between $x$ and $y=p_k(x)$ can be decomposed as $l = \bigsqcup_{ \ 0 \leq I < k} l_i$, where each $l_i$ is a geodesic from $p_i(x)$ to $p_{i+1}(x)$. Moreover, they stay in different bubbles thus their lengths are independent of each other.
		\end{lemma}
		
		\begin{proof}
			Since pinch points are the only possible exit points to get out of a filled bubble, we can easily get the decomposition. See the Figure below for an illustration. Each $l_i$ stays in a single bubble: otherwise, it would pass a pinch point twice and form a loop, which it cannot do since a geodesic must be a simple curve. Finally, by the spatial Markov property (Proposition~\ref{spatial_Markov}) and translation invariance (Lemma~\ref{prop:translation_invariance}), we get the independence.
			
			\begin{figure}[htbp]
				\centering
				\vspace{0.1in}\includegraphics[scale=0.23]{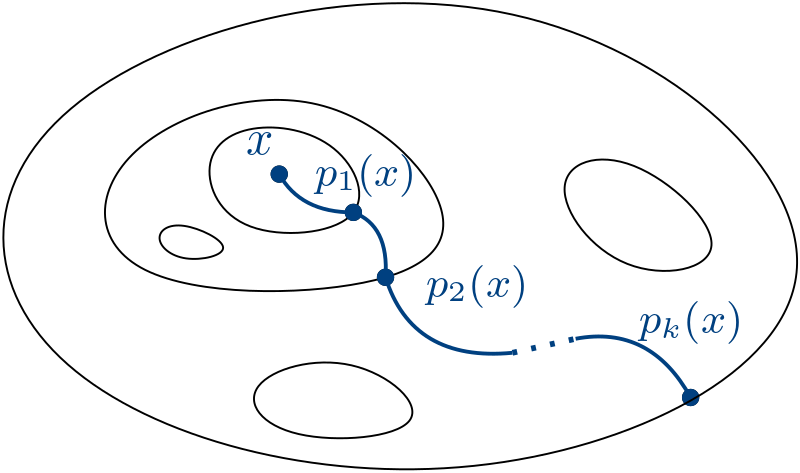}
				\label{fig:decomposition_of_geodesic}
			\end{figure}
			
		\end{proof}
		
		Since $T$ is a subgraph of $M$, we have that $\dM$ is bounded above by $\dT$. The key tool that allows us to get a bound in the other direction is \cite[Proposition 3.8]{she16}\footnote{Terminology ``$\E |E| < \infty$'' in \cite[Proposition 3.8]{she16} is equivalent to $p > 1/2$ by \cite[Lemma 3.1, Lemma 3.6]{she16}.}.
		
		\begin{proposition}[\cite{she16}]\label{finite_expectation}
			Assume that $q > 4$, equivalently, $p > 1/2$. Suppose $[s,t]$ is the smallest among intervals contained $[-1,0]$ such that $X\left(s, t\right)$ is reducible, then $s,t$ are a.s. finite. Assume the reduced form of $X(0,t)$ has $K$ orders, then we have $\mathbb{E}[K]<\infty$. 
		\end{proposition}
		
		From Proposition \ref{finite_expectation}, we can see that when $p>1/2$, a.s.\ for each $x\in\mathbb Z$, there are infinitely many intervals $[a,b]$ with $a < 0 < b$ such that $X(a,b)$ is reducible. Indeed, by the spatial Markov property (Proposition~\ref{spatial_Markov}), if we condition on the first $k$ reducible words containing $o$, then the probability that there are at least $k+1$ reducible words containing $o$ is equal to the probability that there is at least one reducible word containing $o$. This probability is equal to 1 by Proposition~\ref{finite_expectation}.
		
		\begin{proposition} \label{prop:dist-K}
			$d_T(o,p(o)) = K$, where $K$ is defined in Proposition \ref{finite_expectation}.
		\end{proposition}
		
		\begin{proof}
			Let $t$ as defined in Proposition \ref{finite_expectation}. Consider the exploration line passing $\{\hyperref[def:Tri]{\mathrm{Tri}}(i)\}_{ 0 \leq i \leq t}$  in order. It is a contour line of $T$ which starts near $o=\hyperref[def:V and E]{V}(0)$, then takes the right direction along the spanning tree, to the pinch point $p(o)=\hyperref[def:V and E]{V}(t+1)$. When one hamburger is consumed by an order, the exploration line will come back to the root of a tree branch. Thus when we compute $\dT(o,p(o))$, we will exactly omit such a consumed pair. Thus $\dT(o,p(o)) \leq \# \{ \text{Hamburger orders and flexible orders in } \overline{X(0,t)}\}\leq K$. See Figure \ref{fig:dT=K} for an explanation. 
		  \begin{figure}[htbp]
			 \centering
			 \vspace{0.2in}
			 \includegraphics[scale=0.18]{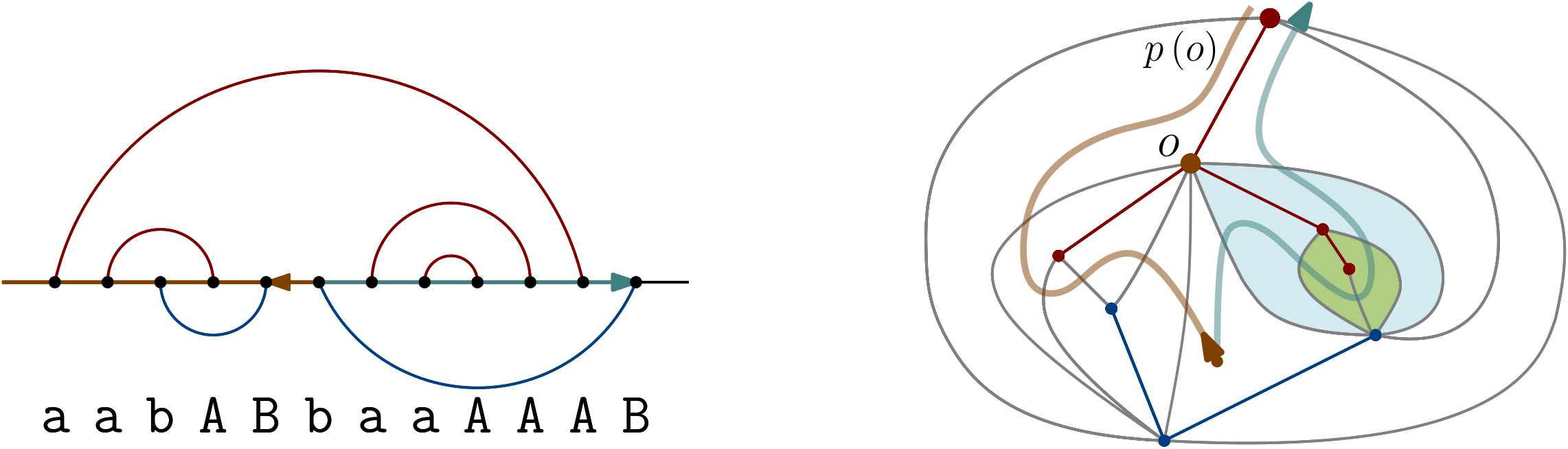}
			 \caption{The blue line is the exploration line in the positive direction, while the brown one is in the negative direction.}
			 \label{fig:dT=K}
    		\end{figure}
        \end{proof}
		
    	By Propositions~\ref{finite_expectation} and~\ref{prop:dist-K},
    	\[
    		      0<\mathbb{E}\dM(o,p(o)) \leq \mathbb{E}\dT(o,p(o)) \leq \E(K) < \infty. 
    	\]
		Let \label{def:alpha}
		\[
		\hyperref[def:alpha]{\mathfrak{a}} := \frac{\mathbb{E}\dT(o,p(o))}{\mathbb{E}\dM(o,p(o))} \in [1,\infty) .
		\]
		The map is invariant under re-rooting, thus we have $\mathbb{E}(\hyperref[def:alpha]{\mathfrak{a}}\dM(x,p(x))-\dT(x,p(x)))=0$ for any $x = \hyperref[def:V and E]{V}(t)$ for a fixed $t$. For simplicity of notation, we define\label{D}
		\[
		\hyperref[D]{D}(x,y) := \hyperref[def:alpha]{\mathfrak{a}} \dM(x,y)-\dT(x,y) .
		\]
		
		\begin{proposition}\label{bound_pinch}
			For $ \varepsilon >0$, let $x$ be the $\lfloor n\varepsilon \rfloor$-th pinch point of $o$, we have:
			$$\frac{1}{n}\hyperref[D]{D}(o,x)\rightarrow 0 \text{ a.s. and in }L^1.$$
		\end{proposition}
		
		\begin{proof}
			Each path in $M$ or $T$ from $o$ to $x$ must visit the $k$-th pinch point of $o$ for each $k=1,\dots,\lfloor n \varepsilon \rfloor$. Let $x_1,...,x_{\lfloor n \varepsilon \rfloor}=x$ be these pinch points, and $o=x_0$. Then we have $\hyperref[D]{D}(o,x)=\sum_{0 \leq i \le \lfloor n \varepsilon \rfloor - 1}\hyperref[D]{D}(x_i,x_{i+1})$. By the spatial Markov property (Lemma \ref{spatial_Markov}) and decomposition of paths (Lemma \ref{lem:Decomposition of the path}), $\hyperref[D]{D}(x_i,x_{i+1})$ only depends on the information inside the corresponding bubble and hence these random variables are i.i.d.\ with expectation $0$. By ergodic theorem, since $\E |\hyperref[D]{D}(x_i, x_{i+1})| < (1+\hyperref[def:alpha]{\mathfrak{a}}) \ \E \dT(o,p(o)) \leq \infty$, we have $\hyperref[D]{D}(o,x)/\lfloor n \varepsilon \rfloor \rightarrow 0 $ a.s. and in $L^1$. Therefore, $\hyperref[D]{D}(o,x)/n$ also tends to $0$ a.s. and in $L^1$.
		\end{proof}
		
	\subsection{Local GHP convergence}
		We will use a similar procedure as Section \ref{chap_crt} to show the local GHP convergence. The first step is to control the uniform difference in metrics. To be more specific, we want to prove:
		\begin{theorem}\label{main_4.2}
			Fix $r>0$, as $n\rightarrow \infty$, $\sup_{x,y \in \hyperref[def:V and E]{V}[-rn^2,rn^2]} \frac{1}{n}|\hyperref[D]{D}(x,y)| \rightarrow 0$ in probability. 
		\end{theorem}
		
		We will deduce Theorem~\ref{main_4.2} from Proposition~\ref{bound_pinch} and a union-bound argument.
		
		\begin{lemma}\label{Lemma_in_4.2}For $r > 0$ and $p\in (0,1)$, there exists a deterministic $R > 0$ such that for each large enough $n$, it holds with probability at least $p$ that $B_{rn}(T) \subset \hyperref[def:V and E]{V}[-Rn^2,Rn^2]$. Similarly, $\exists R > 0$ such that for each large enough $n$, it holds with probability at least $p$ that $\hyperref[def:V and E]{V}[-rn^2,rn^2] \subset B_{Rn}(T)$. 
		\end{lemma} 
		
		\begin{proof}
			Recall that $B_r(T)$ (resp. $B_r(\scale{T})$) denotes the ball of radius $n$ centered at root vertex in $T$ (resp. $\scale{T}$). By Lemma \ref{compact T ball}, we know a.s. $\exists$ \textit{random} $R$ and $N$ {such that} $B_r(\scale{T}) \subset \scaleproj([-R,R])$ for all $n \geq N$. Thus for any $p \in (0,1)$, there exists \textit{deterministic} $R$ and $N$ s.t. $\forall n \geq N$, 
			$$\P\left(B_{rn}(T) \subset \hyperref[def:V and E]{V}[-Rn^2,Rn^2]\right)=\P(B_r ( \scale{T}) \subset \scaleproj([-R,R])) \geq p.$$
			On the other side, by Theorem \ref{conv_to_BM}, we know that $(\frac{1}{n} H_{n^2t})_{-r\leq t \leq r}$ converges to the Brownian motion $(L_t)_{-r\leq t \leq r}$ in distribution with respect to the uniform topology. $L_t$ is a.s. bounded in $[-r,r]$, thus for any $p<1$, there is deterministic $R$, such that $\P\left(|L_t| \leq R/6, \, \forall t \in [-r,r] \right) \geq p$. Thus there is $N>0$ s.t. for all $n \geq N$,
			$$\P(|H_t| \leq nR/3, \, \forall t \in [ -rn^2, rn^2] ) \geq p.$$
			Thus $d_T(o,\hyperref[def:V and E]{V}(t))\leq H_t-2\inf_{-rn^2\leq u \leq rn^2}  H_u$ $\Rightarrow$ $\P(d_T(o,x) \leq R, \, \forall x \in \hyperref[def:V and E]{V}[ -rn^2,  rn^2] ) \geq p$.
		\end{proof}
		
		\begin{proposition}\label{proposition_in_4.2}
			Fix $r>0$ and $\varepsilon>0$ and let  
			\[
			S_n:=\{ \lfloor k \varepsilon n^2 \rfloor: k \in \mathbb{Z}, -r \leq k \varepsilon  \leq r\}, \quad V_n:=\{\hyperref[def:V and E]{V}(t): t \in S_n\}.
			\]
			Then as $n \rightarrow \infty$, $\sup_{x,y \in V_n} n^{-1} |\hyperref[D]{D}(x,y)| \rightarrow 0$ in probability. 
		\end{proposition}
		
		\begin{proof} 
			For any fixed times $s,t \in S_n$, let $x=\hyperref[def:V and E]{V}(s)$, $y=\hyperref[def:V and E]{V}(t)$. Let $x_*$ be the pinch point corresponding to the smallest reducible interval $[a,b]$ such that $\hyperref[def:V and E]{V}[a,b]$ contains $x$ and $y$. Choose $A , B \in \mathbb N$ such that $x_*  = p_{A+1}(x) = p_{B+1}(y)$.
			
			We can decompose each $M$-geodesic from $x$ to $y$ into paths from $p_k(x) $ to $p_{k+1}(x)$ for $k=0,\dots,A -1$, an $M$-geodesic from $p_A(x)$ to $p_B(y)$, and paths from $p_{k+1}(y)$ to $p_k(y)$ for $k=0,\dots,B-1$. The same is true with $T$ in place of $M$. 
			
			The $M$-graph distance and the $T$-graph distance from $p_A(x)$ to $p_B(y)$ are each bounded above by the $T$-graph distance from $p_A(x)$ to $x_* = p_{A+1}(x)$ plus the $T$-graph distance from $p_B(y)$ to $x_* = p_{B+1}(y)$. Therefore,
			\begin{multline}\label{formula_decomposition}
				\frac{1}{n}|\hyperref[D]{D}(x,y)|\leq \frac{1}{n} \left| \sum_{0 \leq k < A}\hyperref[D]{D}(p_k(x),p_{k+1}(x))\right|+\frac{1}{n} \left|\sum_{0 \leq k < B}\hyperref[D]{D}(p_k(y),p_{k+1}(y))\right| 
				\\+\frac{1+\hyperref[def:alpha]{\mathfrak{a}}}{n}\dT(p_A(x),p_{A+1}(x))+\frac{1+\hyperref[def:alpha]{\mathfrak{a}}}{n}\dT(p_B(y),p_{B+1}(y)) .
			\end{multline}
			
			\begin{figure}[htbp]
				\centering
				\includegraphics[scale=0.24]{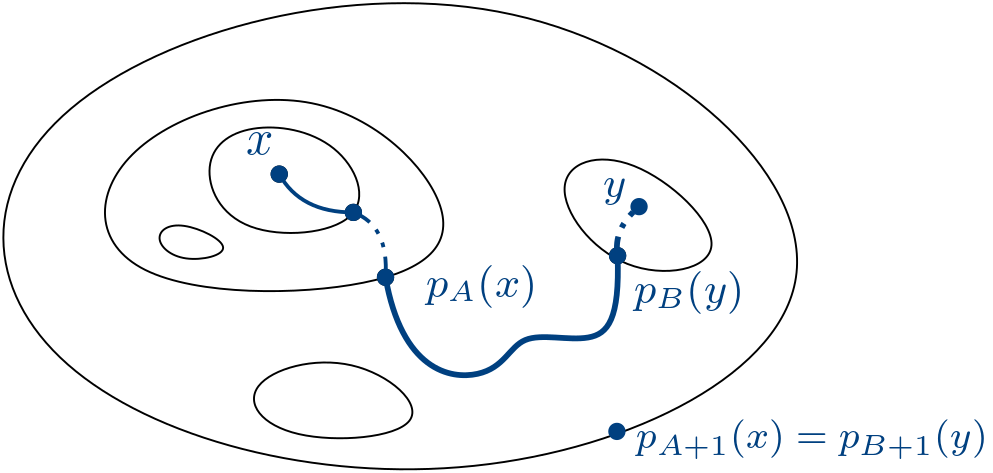}
				\label{fig:distance_between_two_points}
			\end{figure}
            \vspace{0.1in}
			
			For any $p<1$, by Lemma \ref{Lemma_in_4.2}, there is a deterministic $R$, such that for all $n$ large enough, with probability at least $p$, $\hyperref[def:V and E]{V}[-rn^2,rn^2] \subset B_{Rn}(T)$. When $\hyperref[def:V and E]{V}[-rn^2,rn^2] \subset B_{Rn}(T)$, since all of the pinch points between $x$ and $y$ will be visited by the geodesic from $x$ to $y$, the number of \textit{distinct} pinch points between them is smaller than $2Rn$. Note when the one of endpoints of $k$-th and $k+1$-th innermost reducible intervals correspond to a hamburger, $p_k(x)$ and $p_{k+1}(x)$ should be different. By Markov property and large law of numbers, we know with probability tending $1$ that there will be no less than $N/2$ \textit{distinct} pinch points in the first $N$ pinch points, as $N$ goes to infinity. Therefore, the probability that $A,B \leq 4Rn$ tends to $1$ as $n$ goes to infinity.
			
			By translation invariance from Proposition \ref{prop:translation_invariance}, Proposition \ref{bound_pinch} still holds when we replace $o=\hyperref[def:V and E]{V}(0)$ by any other $x=\hyperref[def:V and E]{V}(s)$. Hence we have $\E |n^{-1}{\hyperref[D]{D}(x, p_{4Rn}(x))| } \to 0$ as $n \to \infty$. Since $\hyperref[D]{D}(p_k(x),p_{k+1}(x))$ are i.i.d.\ with $\E \hyperref[D]{D}(p_k(x),p_{k+1}(x))=0$, $\E |\hyperref[D]{D}(p_k(x),p_{k+1}(x))|<\infty$, for any $\delta >0$, we can apply Doob's maximal inequality,
			$$\P\left(\sup_{0\leq i < 4Rn}\frac{1}{n} \left| \sum_{0 \leq k < i}\hyperref[D]{D}(p_k(x),p_{k+1}(x))\right| \geq \delta\right) 
			\leq 
			2\delta^{-1}\E \left|\frac{1}{n}\hyperref[D]{D}(x, p_{4Rn}(x))\right| \to 0.$$
			
			Hence we get the first term on the right side of \eqref{formula_decomposition} converges to $0$ in probability:
		
            \begin{align*}
            	\P\Biggl(\frac{1}{n} \left| \sum_{0 \leq k < A}\hyperref[D]{D}(p_k(x) ,p_{k+1}(x))\right| &\geq \delta\Biggr)\\
            	\leq \P&\left(\frac{1}{n} \left| \sum_{0 \leq k < A}\hyperref[D]{D}(p_k(x),p_{k+1}(x))\right| \geq \delta, A\leq 4Rn\right)+\P(A>4Rn)\\
            	\leq \P&\left(\sup_{0\leq i \leq 4Rn}\frac{1}{n} \left| \sum_{0 \leq k < i}\hyperref[D]{D}(p_k(x),p_{k+1}(x))\right| \geq \delta\right)+\P(A>4Rn) \to 0.
            \end{align*}
            Similarly, the second term on the right side of \eqref{formula_decomposition} converges to $0$ in probability. 
			
			For the last two terms, since $d_T(p_k(x),p_{k+1}(x))$ are i.i.d.\ non-negative random variables with $c:=\E d_T(p_k(x),p_{k+1}(x)) < \infty$, we can apply large law of numbers and Doob's maximal inequality to get that for any $\delta >0$,
			$$\P\left(\sup_{0 \leq k \leq 4Rn} |d_T(x,p_{k}(x))-kc| \geq n\delta\right) \leq \frac{2}{\delta n} \E |d_T(x,p_{4Rn}(x))-4Rnc|\to 0.$$  
            Note when $\sup_{k \leq 4Rn} | d_T(x,p_k(x))-kc| \leq n\delta$, we have for all $k \leq 4Rn$,
            $$ d_T(p_{k-1}(x),p_k(x)) \leq (kc+\delta n)-((k-1)c-\delta n)=c+2\delta n$$  
            Choose $n$ large enough s.t. $n\delta>c$, we get that, as $n \to \infty$,
            $$ \P\left(\sup_{k\leq 4Rn} d_T(p_{k-1}(x),p_k(x)) \leq 3\delta n\right) \leq \P\left(\sup_{k \leq 4Rn} | d_T(o,p_k(x))-kc| \leq n\delta\right) \to 1.$$
            
            Now we can get that the third item in \eqref{formula_decomposition} tends to $0$ in probability:
			$$\P\left(\frac{1}{n}\dT(p_A(x),p_{A+1}(x)) \geq \delta\right)
			\leq \P\left(\frac{1}{n}\sup_{0 \leq k \leq 4Rn} d_T(p_k(x),p_{k+1}(x))\geq \delta\right) + \P(A > 4Rn) \to 0.$$
			It is also true for the last term by symmetry. Thus $\forall \ \delta >0, q>0$, there is $N$ s.t. for any $n\geq N$, 
			$$ \P(|\hyperref[D]{D}(x,y)| \geq n\delta) \leq q, \quad \forall x,y \in M[-rn^2, rn^2].$$
			Note there is constant $C$, s.t $|V_n| \leq C$ for all $n$, we get by a union bound that for all $n \geq N$:
			$$\P(\sup_{x,y \in V_n}n^{-1}|\hyperref[D]{D}(x,y)| \leq \delta) 
			\geq 1-\sum_{x,y \in V_n}\P(n^{-1}|\hyperref[D]{D}(x,y)| \geq \delta)
			\geq 1-C^2q,$$
			which shows $\sup_{x,y \in V_n}n^{-1}|\hyperref[D]{D}(x,y)|$ converges to $0$ in probability.
		\end{proof}
		
    	\begin{proof}[Proof of Theorem \ref{main_4.2}:] 
            Fix $p>0$, $\delta>0$, $r>0$. By Theorem \ref{conv_to_BM}$, (\frac{1}{n} H_{n^2t})_{t\in[-r,r]} \rightarrow (L_t)_{t\in[-r,r]}$ in law with respect to uniform distance. Thus, there is $N=N(r,p,\delta)$, such that for each $n>N$, with probability larger than $1-p$, $$\sup_{t\in[-r,r]} |H_{n^2t}-nL_t| < n\delta.$$ 
            Moreover, for our Brownian motion $L$, there is constant $\varepsilon=\varepsilon(r,p,\delta)>0$, such that with probability at least $1-p$, we have 
            $$|L_s-L_t|<\delta, \ \text{whenever } |s-t|\leq \varepsilon, \ s,t \in [-r,r].$$ 
            Therefore, for each $n>N$, with probability at least $1-2p$,
            $$\sup_{\substack{s,t \in [-r,r] \cap 1/n^2 \mathbb{Z} \\ |s-t|\leq \varepsilon}}d_T(\hyperref[def:V and E]{V}(sn^2),\hyperref[def:V and E]{V}(tn^2)) \leq 4 \sup_{i\in [sn^2,tn^2]} |H_{tn^2} - H_i| \leq 8n\delta+4n\sup_{i \in [s,t]} |L_t-L_i| \leq 12n\delta,$$
            where the first inequality is obtained by the definition of $d_T$.
            
        	Now we use Proposition~\ref{proposition_in_4.2} with $\varepsilon=\varepsilon(r,p,\delta)$ as given above. We get that there is $N_1=N_1(r,p,\varepsilon)=N_1(r,p,\delta)$, such that when $n>N_1$, with probability greater than $1-p$, $$\sup_{x,y \in V_n} \left| \hyperref[D]{D}(x,y) \right|<n\delta.$$ 
            For any fixed times $s,t \in [-rn^2,rn^2]\cap\mathbb{Z}$, there is $s^{\prime},t^{\prime} \in S_n$ such that $|s-s^{\prime}| \leq \varepsilon n^2$,  $|t-t^{\prime}| \leq \varepsilon n^2$. Let $x=\hyperref[def:V and E]{V}(s)$, $y=\hyperref[def:V and E]{V}(t), x^{\prime}=\hyperref[def:V and E]{V}(s^{\prime})$, $y^{\prime}=\hyperref[def:V and E]{V}(t^{\prime})$, we have,
            \begin{align*}         
    			|\hyperref[D]{D}(x,y)|&=\left| \hyperref[def:alpha]{\mathfrak{a}} \dM(x,y)-\dT(x,y)\right| \\
                        &\leq \left|\hyperref[D]{D}(x^{\prime},y^{\prime})\right| + \hyperref[def:alpha]{\mathfrak{a}}\left|dM(x,y)-\dM(x^{\prime},y^{\prime})\right| + \left|\dT(x,y)-\dT(x^{\prime},y^{\prime})\right| \\
                        &\leq \left|\hyperref[D]{D}(x^{\prime},y^{\prime})\right| + \hyperref[def:alpha]{\mathfrak{a}}\left(\d_M(x,x^{\prime})+\dM(y,y^{\prime})\right) + \left(\dT(x,x^{\prime})+\dT(y,y^{\prime})\right) \\
                        &\leq \left|\hyperref[D]{D}(x^{\prime},y^{\prime})\right|+(1+\hyperref[def:alpha]{\mathfrak{a}})\left(\dT(x,x^{\prime})+\dT(y,y^{\prime})\right).
    		\end{align*}

            Therefore, for any $p>0$, $\delta>0$, when $n>N_2=N_2(p,r,\delta):=N+N_1$, with probability $> 1-3p$, 
            $$\sup_{x,y \in \hyperref[def:V and E]{V}[-rn^2, rn^2]} |\hyperref[D]{D}(x,y)| \leq n\delta+(1+\hyperref[def:alpha]{\mathfrak{a}})24n\delta,$$ 
            which implies $\sup_{x,y \in \hyperref[def:V and E]{V}[-rn^2,rn^2]} \frac{1}{n}|\hyperref[D]{D}(x,y)|$ converges to $0$ in law. 
    	\end{proof}
	
    	\begin{corollary}\label{balls_bound}
    		Fix $r>0$, for each $ \varepsilon>0$, with probability tending to 1 as $n\to\infty$, 
    		$$B_{rn}(M) \subset B_{(\hyperref[def:alpha]{\mathfrak{a}} r+\varepsilon)n}(T), \quad \quad B_{rn}(T) \subset B_{(\hyperref[def:alpha]{\mathfrak{a}}^{-1}r+\varepsilon)n}(M).$$
    	\end{corollary}
    	
    	\begin{proof}
    		By Lemma~\ref{compare CRT ball and T ball}, since $d_M \leq d_T$, for any $p<1$, with probability at least $1-p$, there exists $R>0$, s.t. $B_{rn}(M), B_{rn}(T) \subset \hyperref[def:V and E]{V}[-Rn^2, Rn^2]$ for large enough $n$. Then use Theorem \ref{main_4.2} to get that with probability at least $1-2p$, for large enough $n$, 
            $$\sup_{x,y \in \hyperref[def:V and E]{V}[-Rn^2,Rn^2]} \frac{1}{n}|\hyperref[def:alpha]{\mathfrak{a}} \dM(x,y)- \dT(x,y)| \leq \varepsilon.$$
            Especially, for root $o=\hyperref[def:V and E]{V}(0)$, $\sup_{y \in \hyperref[def:V and E]{V}[-Rn^2,Rn^2]} |\hyperref[def:alpha]{\mathfrak{a}} \dM(o,y)- \dT(o,y)| \leq n\varepsilon$.
    	\end{proof}

        \begin{proof}[Proof of Theorem \ref{main_thm_in_introduction}]
    	The remaining procedure will be similar to what we do in Section \ref{chap_crt}. We first define a correspondence of $B_{rn}(M)$ to $B_{\hyperref[def:alpha]{\mathfrak{a}}rn}(T)$ by letting one point correspond to its nearest point in the other ball. Then with the help of Proposition \ref{main_4.2}, Lemma \ref{Lemma_in_4.2} and Corollary \ref{balls_bound}, we can show GHP distance of $n^{-2}B_{rn}(M)$ and $n^{-2}B_{\hyperref[def:alpha]{\mathfrak{a}}rn}(T)$ converges to zero in probability, which is similar to the proof of Proposition \ref{discrete tree}.
    	In the end, by dominated convergence theorem, we have the local GHP distance of $(M,\hyperref[def:alpha]{\hyperref[def:alpha]{\mathfrak{a}}}d_M/n,\mu/n^2,o)$ and $(T,d_T/n,\mu/n^2,o)$ converges to zero in probability. Combined with Theorem \ref{main_chap3}, we can prove Theorem \ref{main_thm_in_introduction}.
        \end{proof}

     \section{FK loops}\label{chap_5}
    	In this chapter, we will discuss the macroscopic behavior of FK loops and prove Theorem \ref{thm:main-chap5-loop}. The basic idea is to find a variant of the ``bubbles'' of the previous subsection with the property that each FK loop is entirely contained in some bubble, then control the diameters of all bubbles. 
    	
    	\begin{definition} 
        We call $X(a,b)$ a \textbf{strong reducible word} if (1) $X(a,b)$ is reducible and (2) $X(b)=\texttt{F}, \ X(a)=\texttt{a}$, $a$ and $b$ are matched.
    	\end{definition}

        The pinch point of $X(a,b)$ is called a \textbf{strong pinch point}, and the region $\hyperref[def:Tri]{\mathrm{Tri}}[a,b]$ encoded by $X(a,b)$ is called a \textbf{strong filled bubble}. when we remove all the contained sub-strong filled bubbles (called \textbf{strong holes}) from this region, what remains is called a \textbf{strong bubble}.

        Note $X(a,b)$ is a \textit{strong reducible} word if and only if $X(a+1,b-1)$ is a \textit{reducible} word or an empty word, in conjunction with $X(b)=$\texttt{F}, $X(a)=$\texttt{a}. Essentially, a strong filled bubble is nothing more than an ordinary filled bubble encased within a quadrilateral face, which features a blue diagonal line. See the left side of Figure \ref{fig:strong bubble}. 
        
        An FK loop is described as \textbf{intersecting} a region within the Tutte map if it crosses some Tutte edge in the \textit{interior} (i.e., we exclude the boundary edges) of that region. Similarly, we consider an FK loop to be \textbf{contained} within a region if all the Tutte edges it crosses are within the interior of that region.

    	\begin{proposition}
    		If an FK loop intersects a strong bubble, then it will be contained in this bubble.
    	\end{proposition}
    	
    	\begin{proof}
    		\begin{figure}[htbp]
    			\centering
    			\includegraphics[scale=0.22]{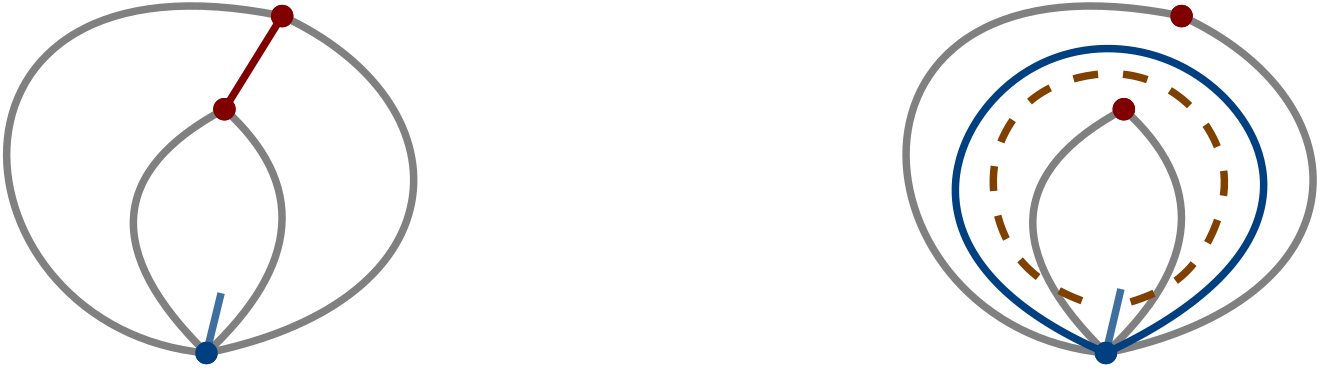}
    			\caption{Recall the \hyperref[Step3 of HC bijection]{\textit{final step}} we construct a planar map from a random walk. For each pair of triangular faces corresponding to \texttt{F} and its matching, we replace the common edge of these two triangles (see the left figure) to its dual (see the right figure).}\label{fig:strong bubble}    
    		\end{figure}
            As explained in previous paragraphs, a blue edge will enclose the inner part of a strong-filled bubble (see the right side of Figure \ref{fig:strong bubble}). An FK loop that intersects with this filled bubble's inner region cannot cross the blue edge. Thus the loop will stay inside the \textit{filled} bubble. This situation also applies to strong holes. If an FK loop intersects with edges outside these holes, it will not cross the blue edge of this hole, thereby totally outside the hole.
    	\end{proof}
    	
    	As a corollary, for an FK loop $\ell$ and the strong bubble $B$ it intersects, we have 
    	\begin{equation}\label{for:5.1}
    		\operatorname{diam} \ell   \leq \sup_{x,y \in B} d_M(x,y)  \leq \sup_{x,y \in B} d_T(x,y) .
    	\end{equation}
        Our method to control the diameter of a strong bubble will be similar the one discussed in Chapter \ref{chap_4.1}.

        For a fixed time $t \in \mathbb{Z}$, let $y$ be the pinch point of $k$-th smallest strong reducible word that the corresponding time interval contains $t$, if it exists. We call $y$ the $k$-th \textbf{strong pinch point} of $x=\hyperref[def:V and E]{V}(t)$, and write $y=s_k(x)$. We also write $s(x)$ for the first strong pinch point and $s_0(x)$ for $x$ itself.

        \begin{lemma}
            $s_{k-1}(x) \neq s_k(x)$.
        \end{lemma}
        \begin{proof}
            Let $X[a,b]$ be the $k$-th strong reducible word of $t$. As we see in Figure \ref{fig:strong bubble}, $s_k(x)=\hyperref[def:V and E]{V}(a)$ will be separated from all other red points $\{\hyperref[def:V and E]{V}(i), a<i\leq b \}$, by the blue edge from $\hyperref[def:V and E]{V^*}(a)$ to itself. 
        \end{proof}

        Let $X[a_i,b_i]$ be the $i$-th \textit{reducible} word of $t$. Set $a_0=t$, $b_0=t-1$. Here, we intentionally make $a_0$ larger than $b_0$ so that $[a_0,b_0]$ is empty. Let $\tau_0=0$, for $k \geq 1$, $$\tau_k:=\min\{ i > \tau_{k-1} \mid a_i-a_{i-1}=-1, b_i-b_{i-1}=1, X(b_i)=\texttt{F}, X(a_i)=\texttt{a} \},$$
        that is, $\tau_k$ is the first time after $\tau_{k-1}$ that $\tau_k$-th reducible word of $x$ is exactly a strong reducible word. From the strong Markov property of the random word, we immediately get the following result:

        \begin{proposition}
            We have $s_k(x)=p_{\tau_k}(x)$, for $k \geq 0$. Moreover, $\tau_k-\tau_{k-1} \overset{d}{=} \tau_1$ are i.i.d variables of geometric distribution with success probability
            $$\P(\tau_1=1)=\P(X(t)=\texttt{F}, X(t-1)=\texttt{a})=\frac{p}{8}.$$
        \end{proposition}
    
    	\begin{corollary}
    	    Let $o$ be the root, then
    		$\E d_T(o,s(o))<\infty.$
    	\end{corollary}
        \begin{proof}
            Note $d_T(p_{k-1}(o), p_{k}(o))$ are i.i.d with finite expectation and $\tau_1$ is a stopping time, by Wald's equation,
            $\E d_T(o,s(o))= \E d_T(o,p_{\tau_1}(o)) =\E \tau_1 \ \E d_T(o,p(o))=8/p \ \E d_T(o,p(o))<\infty.$
        \end{proof}
    	
    	By Corollary \ref{prop:strong_markov}, the spatial Markov property also holds for a strong bubble: When we condition on the $k$-th strong filled bubble of $o$, the remaining part of the map $M$ will have the same distribution as the original map.
    	
    	What's more, a similar large law of numbers argument as in Proposition \ref{bound_pinch} can also be applied in the setting of a strong pinch point, since the key ingredients of the proof are the finite expectation of $d_T(o,s(o))$ and the spatial Markov property. We get:
    	
    	\begin{proposition}\label{prop:large-law-chap5}
    		For $ \varepsilon >0$, let $x$ be the $\lfloor n\varepsilon \rfloor$-th strong pinch point of $o$, we have:
    		$$\frac{1}{n}d_T(o,x)\rightarrow \E d_T(o,s(o)) \ \text{ a.s. and in }L^1.$$
    	\end{proposition}

        \begin{lemma}\label{lem:chap5_distance_between_strong_pinch_points}
            Fix any $\delta >0$, let
            $$E_n:=\{ \forall k \leq n, d_T(s_{k-1}(o),s_{k}(o))<\delta n\}.$$
            Then $\P(E_n) \to 1$, as $n \to \infty$. 
        \end{lemma}
        \begin{proof}
            Let $c:=\E d_T(o,s(o))<\infty$. By Doob's maximal inequality and Proposition \ref{prop:large-law-chap5}, 
            $$ \P\left(\sup_{k \leq n} | d_T(o,s_k(o))-kc| \geq n\delta\right) \leq 2\delta^{-1}\E \left|\frac{d_T(o,s_n(o))}{n}-c\right| \to 0,$$
            Note when $\sup_{k \leq n} | d_T(o,s_k(o))-kc| \leq n\delta$, we have for all $k \leq n$,
            $$ d_T(s_{k-1}(o),s_k(o)) \leq (kc+\delta n)-((k-1)c-\delta n)=c+2\delta n.$$  
            Choose $n$ large enough s.t. $n\delta>c$, we get that, as $n \to \infty$,
            $$ \P\left(\sup_{k\leq n} d_T(s_{k-1}(o),s_k(o)) \leq 3\delta n\right) \leq \P\left(\sup_{k \leq n} | d_T(o,s_k(o))-kc| \leq n\delta\right) \to 1.$$
        \end{proof}

        By translation invariance, we can substitute $o=\hyperref[def:V and E]{V}(0)$ with $x=\hyperref[def:V and E]{V}(t)$ for any deterministic $t$, and analogous results will persist.
    	
    	Next, we will discuss a theorem that, when combined with Inequality (\ref{for:5.1}), reveals the degenerate result Theorem \ref{thm:main-chap5-loop}.
    	
    	\begin{theorem}\label{thm:main_5.1}
    		For deterministic $r>0$, let $\mathcal{B}_n$ be the set of all strong bubbles which intersect $\hyperref[def:Tri]{\mathrm{Tri}}[-rn^2,rn^2]$, then as $n \rightarrow \infty$:
    		$$  \sup_{B \in \mathcal{B}_n} \sup_{x,y \in B}\frac{1}{n} \ d_T(x,y) \rightarrow 0, \quad \text{in probability.} $$    
    	\end{theorem}

        \begin{lemma}\label{lem:intersect_contained_loop}
            For deterministic $r>0$, $\exists \ R=R(r)$ such that  as $n \to \infty$, 
            $$ \P\left( \ \forall B \in \mathcal{B}_n, \ B \subset \hyperref[def:Tri]{\mathrm{Tri}}[-Rn^2,Rn^2] \ \right) \to 1.$$
        \end{lemma}
        \begin{proof}
            By Lemma \ref{Lemma_in_4.2}, there exists deterministic $r_1=r_1(r)$, such that for any $\delta>0$, when $n$ large enough, 
            \begin{equation}\label{eqn:prob1}
               \P\left(\hyperref[def:V and E]{V}[-rn^2,rn^2] \subset B_{r_1 n}(T)\right) \geq 1-\delta.
            \end{equation}
            We also know from Lemma \ref{lem:chap5_distance_between_strong_pinch_points} that, when $n$ large enough, 
            \begin{equation}\label{eqn:prob2}
                \P\left(\sup_{k \leq r_1n} d_T(s_{k-1}(o),s_k(o)) \leq r_1 n\right) \geq 1-\delta.
            \end{equation}
            Assume $\hyperref[def:V and E]{V}[-rn^2,rn^2] \subset B_{r_1 n}(T)$ and $\sup_{k \leq r_1n+1} d_T(s_{k-1}(o),s_k(o)) \leq r_1 n$, let $s_k(o)$ be the first strong pinch point of $o$ which is outside of $B_{r_1 n}(T)$, then 
            $$ r_1 n \geq d_T(o,s_{k-1}(o))=\sum_{i < k} d_T(s_{i-1}(o), s_i(o)) \geq k-1 \Rightarrow k \leq r_1n+1 \Rightarrow d_T(s_{k-1}(o),s_k(o)) \leq r_1 n.$$
            
            Besides, we know all the loops which intersect the inner part of $k$-th filled strong bubble of $o$ will be totally contained in this filled bubble. We also have $d_T(o,s_k(o)) \leq 2 r_1 n$, which means the $k$-th filled strong bubble is contained in $B_{2r_1n}$. Use the proof in Lemma \ref{Lemma_in_4.2} again, we know there is $R=R(r_1)=R(r)$ such that  when $n$ large enough, 
            \begin{equation}\label{eqn:prob3}
            \P\left( \proj_H^{-1}(B_{2 r_1 n}(T)) \subset [Rn^2,Rn^2] \right) \geq \P\left( B_{3 r_1 n}(T) \subset \hyperref[def:V and E]{V}[-Rn^2,Rn^2]\right) \geq 1-\delta.
            \end{equation}
            This implies the endpoints of interval $[a,b]$ corresponding to the $k$-th filled strong bubble are contained in $[-Rn^2, Rn^2]$. So in fact the whole interval $[a,b]$ is contained in $[-Rn^2, Rn^2]$.

            Combined with \eqref{eqn:prob1}, \eqref{eqn:prob2} and \eqref{eqn:prob3}, we know $\exists R=R(r)$, for any $\delta>0$, when $n$ large enough, $\P\left( \forall B \in \mathcal{B}_n, B \subset \hyperref[def:Tri]{\mathrm{Tri}}[-Rn^2,Rn^2]\right)\geq 1-3\delta.$
        \end{proof}
    	
    	\begin{proof}[Proof of Theorem \ref{thm:main_5.1}]
            With the help of Lemma \ref{lem:intersect_contained_loop}, it suffices to prove for any fixed $R>0$,  
            $$ \sup_{B \subset \hyperref[def:Tri]{\mathrm{Tri}}[-Rn^2,Rn^2]} \sup_{x,y \in B} \frac{1}{n} \ d_T(x,y) \to 0 \quad \text{in probability.}$$
            
            We know for any $\delta >0$ and $p>0$, there is $\varepsilon=\varepsilon(R,p,\delta)$, when let event 
            $$\mathfrak{B}_n:=\{ d_T(\hyperref[def:V and E]{V}(s),\hyperref[def:V and E]{V}(t))\leq n \delta, \ \forall |s-t|\leq \varepsilon n^2, s,t \in [-Rn^2,Rn^2] \},$$
            then $\P(\mathfrak{B}_n)\geq 1-p$ when $n$ large enough. Assuming $\mathfrak{B}_n$, consider the event   
            $$\mathfrak{A}_n:=\{\exists  a,b \in [-Rn^2,Rn^2], \hyperref[def:V and E]{V}(a), \hyperref[def:V and E]{V}(b) \in \text{ same strong bubble}, d_T(\hyperref[def:V and E]{V}(a),\hyperref[def:V and E]{V}(b)) > 8 \delta n\}.$$

            \begin{figure}[htbp]
    			\centering  			\includegraphics[scale=0.15]{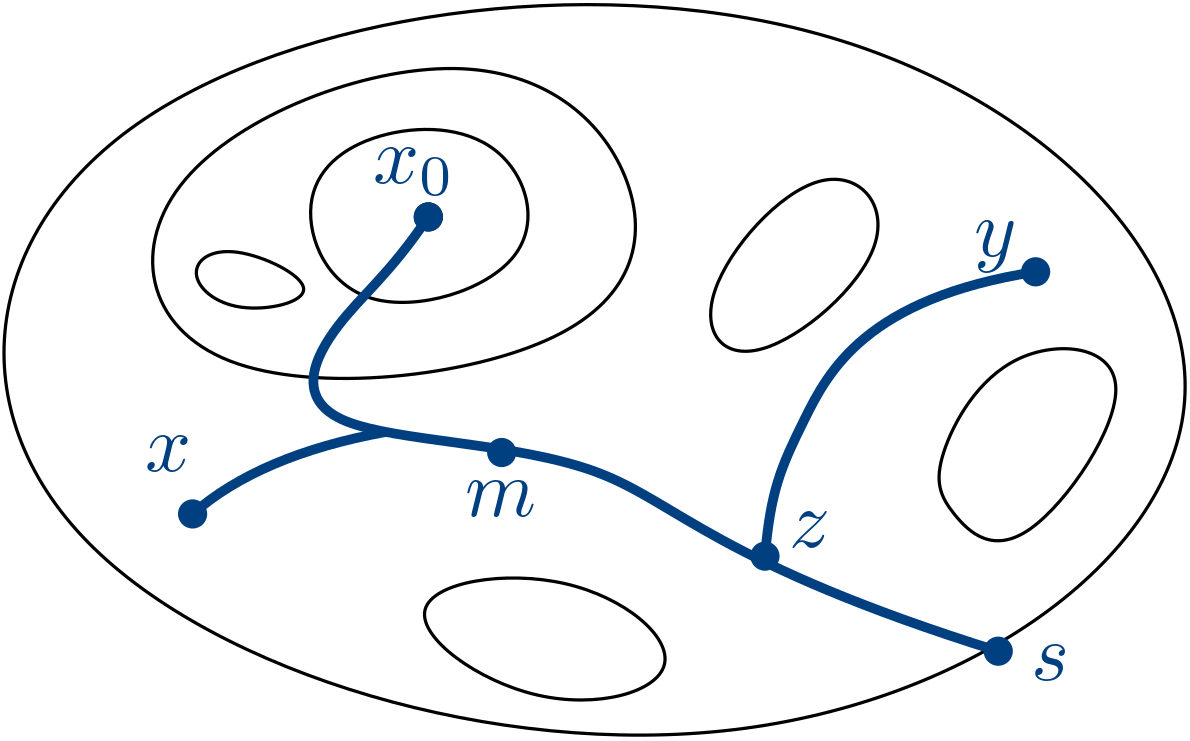}	  \label{fig:distance_in_strong_bubble}
                \vspace{0.1in}
    		\end{figure}
            
            Let $s$ be the pinch point of the common strong bubble of $x:=\hyperref[def:V and E]{V}(a)$ and $y:=\hyperref[def:V and E]{V}(b)$. Suppose the path from $x$ to $s$, from $y$ to $s$ meets at $z=\hyperref[def:V and E]{V}(c)$, where $c \in [a,b]$. By symmetry, we can let $d_T(x,z) \geq (1/2) d_T(x,y) \geq 4 \delta n$. Let $m=\hyperref[def:V and E]{V}(d)$ be the vertex closest to the midpoint in path $x \to z$, where $d \in [a,c]$. Since $d_T(x,m)\geq 2\delta n-1 \geq \delta n$ when $n$ large enough, we know $|a-d| > \varepsilon n^2$ since $\mathfrak{B}_n$. Let $a_0 \in \lfloor \varepsilon n^2 \rfloor \mathbb{Z}$, such that $a<d-\varepsilon n^2 < a_0 \leq d$. Suppose $s$ is $k$-th strong pinch point of $x_0:=\hyperref[def:V and E]{V}(a_0)$. We have $|a_0-d|<\varepsilon n^2$, thus $ d_T(x_0,m) \leq \delta n$, and $k \leq \delta n$. In the meantime, $a_0 \in [a,d]$ and $m,x$ in the same strong bubble implies $m$ is in the path from $s_{k-1}(x_0)$ to $s_{k}(x_0)$. We then have $d_T(s_{k-1}(x_0),s_{k}(x_0)) \geq d_T(m,s)>\delta n$.
            
            Consider set $S_n:=\lfloor \varepsilon n^2 \rfloor \mathbb{Z} \cap [-Rn^2,Rn^2]$ and the event 
            $$\mathfrak{C}_n=: \{ \forall
             t \in S_n, \ \forall i \leq \delta n, d_T(s_i(\hyperref[def:V and E]{V}(t)),s_{i+1}(\hyperref[def:V and E]{V}(t))) \leq \delta n \}.$$
            We know $|S_n|\leq 1+R/\varepsilon$ for all $n$. Thus given $\delta$, $R$, $p>0$ and $\varepsilon=\varepsilon(R,p,\delta)$, by Lemma \ref{lem:chap5_distance_between_strong_pinch_points}, when $n$ large enough, $\P(\mathfrak{C}_n) \geq 1-p$.

            From the discussion above, we know when $\mathfrak{B}_n$ and $\mathfrak{C}_n$ hold, $\mathfrak{A}_n$ is impossible. Thus $\P(\mathfrak{A}_n)\leq (1-\P(\mathfrak{B}_n))+(1-\P(\mathfrak{C}_n))\leq 2p$. To say it alternatively, for any $\delta>0$ and $p>0$, when $n$ large enough,
            $$\P\left(\sup_{B \subset \hyperref[def:Tri]{\mathrm{Tri}}[-Rn^2,Rn^2]} \sup_{x,y \in B} \frac{1}{n} \ d_T(x,y) \geq \delta \right) \leq 2p.$$
    	\end{proof}
    \bibliographystyle{alpha}
    \bibliography{FK_model}

\end{document}